\documentclass{amsart}
 
\usepackage{tikz, tikz-cd}
\usetikzlibrary{matrix,arrows,decorations.pathmorphing}

\usepackage{latexsym, xypic, hyperref, graphicx}
\usepackage{amsfonts}
\usepackage{amssymb}

\newcommand{\R}{\mathbb{R}}

\newcommand{\Z}{\mathbb{Z}}
\newcommand{\C}{\mathbb{C}}
\newcommand{\IP}{\mathbb{P}}

\newcommand{\N}{\mathbb{N}}
\newcommand{\D}{\mathbb{D}}

\def\sfA{{\sf A}}
\def\sfB{{\sf B}}
\def\sfC{{\sf C}}
\def\sfa{{\sf a}}
\def\sfb{{\sf b}}
\def\sfc{{\sf c}}

\newcommand{\rs}{\mbox{$\widehat{\C}$}}

\def\TTT{{\mathcal T}}

\def\GGG{{\mathcal G}}
\def\HHH{{\mathcal H}}

\newcommand{\bdry}{\partial}                     
\newcommand{\id}{\mbox{\rm id}}                  
\newcommand{\Imag}{\mbox{\rm Im}}                    
\newcommand{\Real}{\mbox{\rm Re}}                    

\newcommand{\Aut}{\mbox{\rm Aut}}                        
\newcommand{\Mod}{\mbox{\rm Mod}}    
\newcommand{\intersect}{\cap}                    
\newcommand{\union}{\cup}                        

\newcommand{\mtwo}[4]                            
{\mbox{$\left(\begin{array}{cc}                  
#1 & #2 \\
#3 & #4 
\end{array}
\right)$}}

\newcommand{\dettwo}[4]                          
{\mbox{$\left|\begin{array}{cc}                  
#1 & #2 \\
#3 & #4 
\end{array}
\right|$}}

\newcommand{\pf}{\noindent {\bf Proof: }}

\newcommand{\be}{\begin{enumerate}}
\newcommand{\eb}{\end{enumerate}}
\newcommand{\bi}{\begin{itemize}}
\newcommand{\ib}{\end{itemize}}
\newcommand{\bl}{\begin{list}}
\newcommand{\lb}{\end{list}}

\newcommand{\gap}{\vspace{5pt}}                 

\newcommand{\Fix}{\mbox{\rm Fix}}

\newcommand{\Homeo}{\mbox{Homeo}}

\newcommand{\genby}[1]{\mbox{$\langle #1 \rangle$}}

\newcommand{\Crit}{\mbox{\rm Crit}}

\newtheorem{proposition}{\bf Proposition}

\newtheorem{theorem}{\bf Theorem}
\newtheorem{lemma}{\bf Lemma}
\newtheorem{corollary}{\bf Corollary}

\usepackage{color}

\newcommand{\portrait}{\mathbf{CP}}
\newcommand{\TM}{\mathbf{TM}}
\newcommand{\ABC}{\mathbf{ABC}}
\newcommand{\ABD}{\mathbf{ABD}}
\newcommand{\BC}{\mathbf{BC}}  
\newcommand{\BD}{\mathbf{BD}}

\newcommand{\PMod}{\mbox{\rm PMod}} 
\renewcommand{\P}{\mathbb{P}}

\title{On the classification of critically fixed rational maps}


\begin{author}[K.~Cordwell]{Kristin Cordwell}
\email{kcordwell@gmail.com}
\address{ %
  K.~Cordwell,360 W 43rd St, Apt S8E, 
New York, NY 10036}
\end{author}

\begin{author}[S.~Gilbertson]{Selina Gilbertson}
\email{sjg74@nau.edu}
\address{Selina Gilbertson, 400 E McConnell Dr \#72, Flagstaff AZ  86001}
\end{author}

\begin{author}[N.~Nuechterlein]{Nicholas Nuechterlein}
\email{nknuecht@umich.edu}
\address{711 Catherine St., 
Ann Arbor, MI 48104}
\end{author}

\begin{author}[K.~M.~Pilgrim]{Kevin M. Pilgrim}
\email{pilgrim@indiana.edu}
\address{Dept. Math., Indiana University, Bloomington, IN 47405}
\end{author}

\begin{author}[S.~Pinella]{Samantha Pinella}
\email{S.Pinella@sms.ed.ac.uk}
\address{1713 Tanglewood Dr., 
Loveland, OH 45140}
\end{author}

\begin{document}

\maketitle

\begin{center}
\today
\end{center}

\begin{abstract}
We discuss the dynamical, topological, and algebraic classification of rational maps $f: \rs \to \rs$ each of whose critical points $c$ is also a fixed-point of $f$, i.e. $f(c)=c$.\\
\end{abstract} 

\tableofcontents 

\section{Introduction}

Fix an integer $d \geq 2$, and let $f(z)=\frac{p(z)}{q(z)} \in \C(z)$ be a rational function of degree $d$. Regarded as a self-map of the complex projective line $f: \IP^1 \to \IP^1$, its iterates yield a holomorphic dynamical system. Let 
\[ \C(f):=\{c_1, \ldots, c_n\} \]
denote the set of critical points of $f$, i.e points at which $f$ is not locally injective.  The orbits of these few critical points under iteration of $f$ significantly influence the global dynamical behavior of $f$; see e.g. \cite{milnor:dynamics}. Counted with multiplicity, there are $2d-2$ critical points.  The set $\Fix(f)$ of fixed-points, counted similarly with multiplicity, has $d+1$ elements.

In this article, we consider the classification of those rational maps $f$ for which $\C(f) \subset \Fix(f)$, so that each critical point of $f$ is also a fixed-point of $f$.  Our investigations generalize those of Tischler \cite{tischler:smale}, who considered the special case of polynomials.  

Each  critically fixed rational map $f$ is also a member of the larger set of so-called {\em Thurston maps} \cite{DH1}.  Thurston maps, by definition, are orientation-preserving branched coverings $f: S^2 \to S^2$ of degree at least two for which the set 
$Z(f):=\union_{n \geq 0} f^n\C(f)$ is finite.  A crude, set-theoretic invariant of a Thurston map is its 
critical orbit portrait 
\[ (Z(f) \stackrel{f}{\longrightarrow} Z(f), Z(f) \stackrel{\deg(f, \cdot)}{\longrightarrow} \N).\]  This invariant records the dynamics and local degree of $f$ on the set $Z(f)$. One may also consider abstract critical orbit portraits 
$(Z \to Z, Z(f) \stackrel{\deg()}{\longrightarrow} \N)$, where points $z \in Z$ with $\deg(z)>1$ are regarded as ``critical'', and where suitable admissibility conditions are satisfied; see \S \ref{subsecn:cop} for the precise formulation. 

Recall that the multiplicity $m$ of a critical point $c$ is given by $m=\deg(f,c)-1$. Our first main result is 

\begin{theorem}
\label{thm:portrait_characterization}
Suppose $(Z \to Z, Z \to \N)$ is a critical orbit portrait of degree $d \geq 2$ consisting of $n$ fixed critical points $c_1, \ldots, c_n$, each of multiplicity $m_i \leq d-1$, and satisfying the Riemann-Hurwitz relation $\sum_{i=1}^n m_i=2d-2$.

Then there is a rational function realizing this critical orbit portrait if and only if $n \leq d$. 
\end{theorem}

The necessity of Theorem \ref{thm:portrait_characterization} follows easily from the Holomorphic Fixed-Point Formula. 
We prove sufficiency by applying Thurston's topological characterization of rational functions among Thurston maps \cite{DH1} to the set of Thurston maps obtained by certain surgery procedure on a planar multigraph.   
For our purposes, a {\em planar multigraph $\GGG$} is a simple graph $S\GGG$ embedded in the plane, without loops or multiple edges, and with edges given a positive integer weight indicating the number of edges joining the two incident vertices. The ``blowing up'' surgery of \cite{kmp-tan:blow} assigns, to each such multigraph $\GGG$, a Thurston map $F_\GGG$.   Moreover, $F_\GGG$ is combinatorially equivalent to a rational map $f_\GGG$ if and only if $\GGG$ is connected (\cite[\S 5.2]{kmp-tan:blow}).  By construction, the set of vertices $V(\GGG)$ is in bijective correspondence with the set of critical points $\Crit(f_\GGG)$; the valence of a vertex is the multiplicity of the corresponding critical point.

Sufficiency is then deduced from the following purely-graph theoretic result:

\begin{theorem}
\label{thm:graphtheory}
Fix $d \geq 3$.  Any partition
\[ 2d-2 = m_1 + \ldots + m_n\]
with $n \leq d$ and $1 \leq m_i \leq d$ for each $i=1, \ldots, n$ arises as the set of valences of a connected planar multigraph.
\end{theorem} 

The method of proof yields a new proof of the topological realizability of branch data satisfying the hypotheses in the first paragraph of Theorem 1; see Theorem \ref{thm:genrealizable}.  

We now consider the assignment $\GGG \mapsto F_\GGG$, regarded as a map from planar multigraphs to Thurston maps. 

\begin{theorem}
\label{thm:injective}
Multigraphs $\GGG_1, \GGG_2$ are planar isomorphic if and only if the Thurston maps $F_1, F_2$ are combinatorially equivalent.  
\end{theorem}

Theorem \ref{thm:injective} implies that the planar isomorphism type of $\GGG$ is an invariant of the Thurston class of $F_\GGG$.
Unfortunately we do not know if the function $\GGG \mapsto f_\GGG$ from the set of connected planar multigraphs to the set of M\"obius conjugacy classes of critically fixed rational maps is surjective.   We establish some partial results.  
First, if $p$ is a critically fixed complex polynomial, then $p=P_\GGG$ for some graph $\GGG$ (see \S \ref{subsecn:star}); indeed, in this case $\GGG$ is a star and is in some sense dual to the bipartite tree considered by Tischler.  For maps up to degree five, we find all critically fixed rational maps by an exhaustive algebraic search; comparing the resulting set of solutions with the set of maps obtained by blowing up graphs, we deduce that these two classes are the same; see \S \ref{secn:examples}.  
Making use of a recent topological result of Liu and Osserman \cite{liu:osserman:irreducible}, we will establish the following weaker result. 

\begin{theorem}
\label{theorem:aretwists}
Given two critically fixed Thurston maps $F_1, F_2$, if their number and multiplicities of critical points coincide, then up to combinatorial equivalence, we have $F_2=h\circ F_1$ where $h: (S^2, P(F_1)) \to (S^2, P(F_1))$ is an orientation-preserving homeomorphism fixing $P(F_1)$ pointwise. 
\end{theorem}

Put another way, the {\em Hurwitz class} of a critically fixed Thurston map depends in the sense of S. Koch \cite{koch:criticallyfinite} only on its {\em branch data}. In particular, up to equivalence, every critically fixed rational map $R$ arises, by postcomposing with a homeomorphism, or {\em twisting}, a rational map $R_\GGG$ obtained via the graph construction.  
However, there are explicit examples of distinct graphs with the same branch data. 
We conclude that the graph $\GGG$ is not invariant under twisting; indeed, this is not surprising: the rabbit and airplane polynomials have different Hubbard trees, but are related by twisting; cf. \cite{bartholdi:nekrashevych:twisted}.  

Suppose now $f$ is an arbitrary critically fixed rational map.  Associated to $R$ is another graph-theoretic invariant, which we call its {\em Tischler graph} $\TTT$, defined as the union of the closures of its finite many fixed internal rays; see \S \ref{secn:tischler}.      If $f=f_\GGG$ is obtained by blowing up $\GGG$ then we show that one can compute $\TTT$ from $\GGG$ and vice-versa.   Our lack of understanding regarding whether $f=f_\GGG$ in general stems in part from a lack of knowledge regarding the structure of $\TTT$; in particular, we do not know whether $\TTT$ is always connected.  

Our investigations concern a very special family of dynamical systems. Within this family, a number of natural invariants coincide.   To highlight the central issues, we take here the opportunity to formally elucidate the various relationships that exist between these invariants in as general a context as possible. One reason is that a number of these invariants have close analogs in non-dynamical settings. The distinctions, however, are subtle and important for applications in dynamics. The finest such invariant is a group-theoretic one called a {\em wreath recursion} on the fundamental group $\pi_1(S^2-P(f))$.   Laurent Bartholdi \cite{bartholdi:program:img} has implemented the  iterative algorithm on Teichm\"uller space used in the proof of Thurston's characterization of rational functions.  Taking as input a wreath recursion, it returns as output a numerical approximation of the corresponding rational map, if it exists, or an algebraic description of a topological obstruction. 

Our investigations conclude with systematic enumeration of examples for degrees up to 5, for which naive algebraic methods are sufficient, if a bit unwieldy in some cases.  For maps of degree 6, all such planar graphs were enumerated, and their wreath recursions computed; we present some representative data and discuss these computations in detail.  Conceivably, there are critically fixed maps $g$ which do not arise as maps of the form $f_\GGG$ for some planar multigraph.  In light of Theorem \ref{theorem:aretwists}, $g$ is a twist of some $f_\GGG$.  Future work will address the classification of twists of critically fixed rational maps.  

\subsection{Acknowledgement} This research was supported by NSF DMS REU grants 0851852 and 1156515.  K. Pilgrim was also supported by Simons Foundation grant no.245269 and by the Department of Mathematics of Indiana University.   Fractal images created with {\em FractalStream}, at \url{https://code.google.com/p/fractalstream/downloads/list}. The authors thank L. Bartholdi for sharing and installing his software \cite{bartholdi:program:img} and for related technical assistance.  

\newpage

\section{Thurston maps and their invariants} 
\subsection{Summary of relationships} 
\label{subsecn:summary} 
A {\em Thurston map} is an orientation-preserving self-branched covering $f: S^2 \to S^2$ for which the postcritical set $P(f): = \union_{n>0}f^n(C(f))$ is finite, where $C(f)$ denotes the set of points $x$ at which the local degree $\deg(f,x)$ is larger than $1$. 
Two Thurston maps $f$, $g$ are called {\em Thurston}, or {\em combinatorially} equivalent if there exist orientation-preserving homeomorphisms $h_0, h_1: (S^2,P(f)) \to (S^2, P(g))$ such that $h_0 \circ f = g \circ h_1$ and $h_0$ is isotopic to $h_1$ through homeomorphisms $h_t, 0 \leq t \leq 1$, for which $h_t|_{P(f)} = h_0|_{P(f)}$ for all $t \in [0,1]$. 

It is useful to think of a Thurston map as a noninvertible analog of a homeomorphism of a surface to itself, and of a combinatorial equivalence  class as the analog of a conjugacy class of the corresponding mapping class. 

Let $f: S^2 \to S^2$ be a Thurston map. Associated to $f$ will be five other objects, organized as follows.

\[
\xymatrix{
\mbox{Thurston map } \ar[r] \ar[d] & \mbox{Critical orbit portrait}\ar[d]  \\ 
\mbox{Augmented branched cover} \ar[r] \ar[d] & \mbox{Augmented branch data} \ar[d] \\
\mbox{Branched cover} \ar[r] & \mbox{Branch data}
}
\]

\gap

The top row consists of dynamical objects. The bottom row consists of classical, non-dynamical objects. The middle row is a mild enhancement of the non-dynamical objects. The right-hand column consists of set-theoretic objects, while the left-hand column consists of objects belonging to the realm of the topology of surfaces.  The horizontal arrows arise by restricting to finite sets. The upper vertical arrows forget the identification of domain and range. The  lower vertical arrows arise by forgetting the distinguished role played by a finite set. 

For each type of object, we will have a corresponding notion of isomorphism and of automorphism, and this correspondence will be natural with respect to the implied maps above. We will thereby obtain six categories (whose morphisms are all isomorphisms) and corresponding functors 

\[
\xymatrix{
\TM  \ar[r] \ar[d] & \portrait \ar[d]  \\ 
\ABC  \ar[r] \ar[d] & \ABD  \ar[d] \\
\BC  \ar[r] & \BD \\
}
\]
\gap
Now suppose $f$ is a Thurston map. 

 \subsection{Distinguished sets} 
Naturally associated to a Thurston map $f$ are the following distinguished subsets:
\be
\item the set  $C(f)= \{x : \deg(f, x) > 1\}$ of critical (also called branch) points; here $\deg(f,x)$ is the local degree of $f$ at $x$;

\item the set $V(f):=f(C(f))$, the set of critical (also called branch) values; 

\item the postcritical set $P(f): = \union_{n>0}f^n(C(f))$; by definition, this set is finite for a Thurston map.  We also denote this by $B(f):=P(f)$ when forgetting dynamics (see below). 

\item $Z(f):=P(f) \union C(f)$;

\item $A(f):=P(f) \intersect f^{-1}(B(f))$.
\eb

\subsection{Critical orbit portrait}  
\label{subsecn:cop}

The {\em critical orbit portrait}\footnote{called a {\em mapping scheme} in \cite{kmp:census}}  associated to $f$ is the pair \[\big(\ Z(f) \stackrel{f}{\longrightarrow }Z(f), Z(f) \stackrel{\deg(f, \cdot)}{\longrightarrow}\N \big)\] consisting of the restrictions of $f$ and of the local degree function to the set $Z(f)$; it is therefore a dynamical, set-theoretic invariant. 
We say $\portrait(f)$ and $\portrait(g)$ are {\em isomorphic} if there is a conjugacy $Z(f) \to Z(g)$ respecting the local degree functions.  We obtain a corresponding category $\portrait$ of whose objects are abstract critical orbit portraits $(Z \to Z, \ Z \to \N)$ and whose isomorphisms are defined similarly.  

Note that if $f, g$ are Thurston equivalent via $(h_0, h_1)$ then the restriction $Z(f) \stackrel{h_1}{\longrightarrow} Z(g)$ induces an isomorphism between the corresponding critical orbit portraits. So we obtain a functor $\TM \to \portrait$. 

\subsection{Augmented branched cover} 
\label{subsecn:augbc}

The {\em augmented branched cover} associated to $f$ is the pair $f: (S^2, A(f)) \to (S^2, B(f))$, where now domain and range are no longer identified.  

The augmented branched coverings associated to Thurston maps $f$ and $g$ are isomorphic (or augmented-Hurwitz-equivalent)\footnote{In \cite{koch:criticallyfinite}, this notion is formulated only for maps $f$ and $g$ with $A(f)=A(g), B(f)=B(g)$, and is there called $(A,B)$-Hurwitz-equivalent}  if there are orientation-preserving homeomorphisms $h_0: (S^2, B(f)) \to (S^2, B(g))$ and $h_1: (S^2, A(f)) \to (S^2, A(g))$ such that $h_0 \circ f = g \circ h_1$.  We obtain a corresponding category $\ABC$ of abstract augmented branched coverings $(S^2, A) \to (S^2, B)$. 

Clearly, a Thurston equivalence yields an augmented Hurwitz-equivalence. We obtain in this way a forgetful functor $\TM \to \ABC$. 

Note that if $f, g$ are $(A,B)$-Hurwitz equivalent via a pair $(h_0, h_1)$, then necessarily $A(f) \union C(f) \stackrel{h_1}{\longrightarrow} A(g) \union C(g)$. 
\gap

\subsection{Augmented branch data} 

The {\em augmented branch data} associated to $f$ is the pair \[\left( \ A(f)\union C(f) \stackrel{f}{\longrightarrow} B(f), \ A(f)\union C(f)\stackrel{\deg(f,\cdot)}{\longrightarrow}\N\  \right)\] obtained by restricting $f$ and the local degree function to the set $A(f)\union C(f)$. Here, domain and range of $f$ are no longer identified, so that $A(f)\union C(f)$ and  $B(f)$ are disjoint.  

The augmented branch data of $f$ and of $g$ are called isomorphic if there are bijections $A(f) \union C(f) \to A(g) \union C(g)$ and $B(f) \to B(g)$ such that $A(f) \union C(f) \to A(g) \union C(g)$ respects the local degree functions and the diagram 
\[
\xymatrix{
A(f)\union C(f)  \ar[r] \ar[d]_f & A(g)\union C(g) \ar[d]^g  \\ 
B(f)  \ar[r] & B(g) \\
}
\]
commutes. 
We obtain similarly a corresponding category $\ABD$ of abstract augmented branch data $(A \union C \to B, A \union C \to \N)$. 
Sending an abstract augmented branched cover to the corresponding abstract augmented branch data is a functor $\ABC \to \ABD$. 
\gap

\subsection{Branched cover} 

The {\em branched cover} associated to $f$ is the map $f: S^2 \to S^2$, where now domain and range are no longer identified. 

The branched covers $f, g$ are {\em Hurwitz equivalent} if there are orientation-preserving homeomorphisms $h_0, h_1$ satisfying $h_0 \circ f = g \circ h_1$. Note that $h_1: C(f) \to C(g)$ and $h_0: V(f) \to V(g)$.  Thus (abstract, as opposed to arising from Thurston maps---but see below) branched covers and Hurwitz equivalences form a category $\BC$. 

An isomorphism of augmented branched covers  gives a Hurwitz equivalence, so we again obtain a forgetful functor $\ABC \to \BC$. 
\gap

\subsection{Branch data, or passport} 
\label{subsecn:branchdata}
The {\em branch data} associated to $f$ is the pair \[ \big( \ C(f) \stackrel{f}{\longrightarrow}V(f), \ C(f) \stackrel{\deg(f, \cdot)}{\longrightarrow} \N \ \big).\]
More conventionally, branch data can be regarded as a collection of partitions of the degree $d$ of $f$, one for each element $v$ of $V(f)$, with each partition given by the unordered set of local degrees at points in $f^{-1}(v)$.  
The branch data of $f$ and $g$ are isomorphic if there are bijections $C(f) \to C(g), V(f) \to V(g)$ respecting the local degree functions and satisfying 
\[
\xymatrix{
C(f)  \ar[r] \ar[d]_f & C(g) \ar[d]^g  \\ 
V(f)  \ar[r] & V(g) \\
}
\]
Since in this work we are restricting our attention to self-maps of spheres, it is convenient to now introduce three additional admissibility requirements that must be satisfied by abstract branch data $(C \to V, C \stackrel{\deg }{\longrightarrow} \N)$ which are necessary in order to arise from a branched covering $S^2 \to S^2$.  
\bi 
\item $\deg(c) \geq 2$ for each $c \in \C$, and $C \to V$ is surjective (else we are dealing with augmented branch data)
\item the {\em Riemann-Hurwitz formula} 
\begin{equation}
\label{equation:RH}
\sum_{c \in C}(\deg(c)-1)=2(d-1)
\end{equation}
is satisfied for some integer $d \geq 2$, called the {\em degree}. 
\item $\deg(c) \leq d$ for all $c \in C$. 
\ib
Thus we have a category $\BD$ whose objects are abstract admissible pairs $(C \to V, C \to \N)$. 
We obtain similarly a functor $\BC \to \BD$. By forgetting values off the critical locus (that is, by restricting) we obtain a functor $\ABD \to \BD$. 

\subsection{Hurwitz factorizations} 
\label{subsecn:hurwitzfactor}
There are necessary and sufficient algebraic conditions, algorithmically implementable, for determining Hurwitz equivalence and for determining the realizability of a given abstract admissible branch data by a branched covering; see also \S \ref{subsecn:bc_to_bd} below. Though classical, there are some subtleties regarding identifications, and this seems not to be widely known among researchers in dynamics. 
So we briefly review the theory, due to Hurwitz; see e.g. \cite{berstein:edmonds:construction} . 

Fix $d \geq 2$ and fix an abstract admissible branch data $(C \stackrel{f}{\longrightarrow}V, C \stackrel{\deg()}{\longrightarrow}\N)$ of degree $d$; set $n:=\#V$. Consider the (possibly empty) set of all branched covers $F: S^2 \to S^2$ realizing this branch data. We will assign to $F$ a group-theoretic invariant, called a {\em Hurwitz factorization}.  We will characterize the set of invariants that arise, and we will give necessary and sufficient conditions in terms of Hurwitz factorizations for two branched covers $F, G$ realizing this branch data to be Hurwitz equivalent.  Loosely speaking, this is a coordinate-free classification of isomorphism classes of covering spaces; the base is allowed to vary. 

The invariants we will consider, called {\em Hurwitz factorizations}, take the following form. Fix an enumeration $V=\{v_1, \ldots, v_n\}$. For each $i=1, \ldots, n$ let 
\[ d_{i,1} + \ldots + d_{i, m_i}=d\]
be the partition of $d$ determined from the given branch data by the local degrees $d_{ij}:=\deg(c_{ij})$ above $v_i$. A {\em Hurwitz factorization of degree $d$ (and genus $0$)} is an ordered $n$-tuple of permutations 
\[ (\sigma_1, \ldots, \sigma_n)\]
satisfying the following conditions:
\be
\item (transitivity) the subgroup of $S_d$ generated by $\sigma_1, \ldots, \sigma_n$ is transitive; 
\item (local degrees) for each $i=1, \ldots, n$, 
\[ \sigma_i = \sigma_{i1}\cdot \ldots \cdot \sigma_{im_i}\]
where the $\sigma_{ij}$ are pairwise disjoint (and thus commute, so the order in the above product is irrelevant), and $\sigma_{ij}$ is a cycle of length $d_{ij}$;
\item (factorization) $\sigma_1\cdot \ldots \cdot \sigma_n = \id$; here $\sigma_1$ is performed first, since monodromy is usually a right action. 
\eb

The {\em Hurwitz Existence Theorem} asserts that the given branch data is realizable precisely when there exists a Hurwitz factorization satisfying the above properties; cf. \cite{berstein:edmonds:construction}.   

We turn now to the extent to which the set of such Hurwitz factorizations may be regarded as invariants of branched coverings realizing the given branch data. 

For convenience, fix an identification $S^2=\rs$. Fix an identification $V=\{v_i:=\exp(2\pi i k/n), k=1, \ldots, n\} \subset S^2$.  Let $D:=\{|z|\leq 2\}\subset \rs$.  Let $\Homeo^+(S^2_n)$ and $\Homeo^+(D_n, \bdry D)$ denote respectively the group of orientation-preserving homeomorphisms of the sphere and of the disk $D$ which preserve (not necessarily pointwise) the set $V$ and which in the latter case are the identity on $\bdry D$; each is equipped with the compact-open topology. Let $\Mod(S^2_n), \Mod(D_n)$ denote the corresponding mapping class groups, that is, the group of path-components.  By ``capping'' via an extension fixing the point at infinity, there is a natural inclusion $\Homeo^+(D_n, \bdry D) \hookrightarrow \Homeo^+(S^2_n)$.  By the Alexander Trick, this is surjective on the level of path-components, so the induced homomorphism $\Mod(D_n) \to \Mod(S^2_n)$ is surjective.   The kernel is infinite cyclic if $n \geq 3$;  a generator is represented by a Dehn twist about the curve $|z|=3/2$.   Note that $\Mod(D_n)$ is isomorphic to the Artin braid group $B_n$ on $n$ strands. 

For $i=1, \ldots, n$ let $r_i$ be the Euclidean segment joining the origin to $v_i$, and let $\gamma_i$ be the loop in $\rs \setminus V$ based at the origin which runs along $r_i$, loops around $v_i$ in the counterclockwise direction, and returns along $r_i$. Let $g_i$ be the corresponding element of $\pi_1(\rs\setminus V, 0)$. 

Now suppose $F$ is a branched covering realizing the above branch data. By postcomposing with a (non-unique) homeomorphism $\phi$, we may assume (i) $V(F)=V$, and (ii) above each $v_i \in V$, the partition of the integer $d$ given by the ramification in the given abstract branch data  coincides with that given by the ramification of $F$.
Fix an identification $\Lambda: F^{-1}(0)=\{1, \ldots, d\}$.  Then there is a monodromy representation 
\[ \rho_F: \pi_1(\rs \setminus V, 0) \to S_d.\]
The {\em Hurwitz factorization} $\sigma(F)$ associated to $F$ is the ordered $n$-tuple of permutations 
\[ \sigma(F):=(\sigma_1(F), \ldots, \sigma_n(F)) \in S_d \times \ldots \times S_d=:(S_d)^n \]
where $\sigma_i(F)=\rho_F(g_i)$. 

The Hurwitz factorization associated to $F$ constructed in the previous paragraph depends on (i) the choice of homeomorphism $\phi$, and (ii) the choice of identification $\Lambda$. Clearly, a different choice of $\Lambda$ changes $\rho_F$ by post-composition with an inner automorphism of $S_d$.  Considering  (ii), then, we see that we must consider two Hurwitz factorizations $(\alpha_1, \ldots, \alpha_n), (\beta_1, \ldots, \beta_n)$, to be equivalent if they are simultaneously conjugate, i.e. if there exists $x \in S_d$ such that 
\[(\beta_1, \ldots, \beta_n)=(\alpha_1^x, \ldots, \alpha_n^x)\]
where we denote $g^x:=x^{-1}gx$ and where $x^{-1}$ is performed first (since monodromy is usually a right action). 

If $\phi_1, \phi_2$ are two different homeomorphisms, the resulting coverings branched over $V$ differ by postcomposition by an element $h$ of $\Homeo^+(S^2_n)$. If $F, G$ are two coverings with $V(F)=V(G)=V$ and $G=h\circ F$ where $h \in \Homeo^+(S^2_n)$ then $G^{-1}=F^{-1}\circ h^{-1}$.  It follows that the monodromy representations of $F$ and $G$ are related: 
\[ G = h \circ F \implies \rho_G \dot{=} \rho_F \circ h_* \iff \rho_G \circ h^{-1}_* \dot{=} \rho_F\]
where $\dot{=}$ denotes equality up to postcomposition with inner automorphisms. Note that if the natural map 
\[ \Aut(C \stackrel{f}{\longrightarrow}V, C \stackrel{\deg}{\longrightarrow}\N) \longrightarrow \Aut(V)\]
(obtained by restricting an automorphism of branch data to the set of critical values) has trivial image, then the corresponding homeomorphism $h$ fixes $V$ pointwise, and so the corresponding mapping class (or any corresponding braid) element is pure. 
The effect of this ambiguity on the Hurwitz factorizations can be calculated as follows. Since we consider factorizations differing by simultaneous conjugation to be equivalent, we may assume $h(0)=0$, so that $h_*$ is well-defined. Then for each $i=1, \ldots, n$, 
\[ \sigma_i(G) = \rho_F(h_*^{-1}(g_i)).\]
More concretely: to find $\sigma_i(G)$, express $h_*^{-1}(g_i)$ as a word in the generators $g_1, \ldots, g_n$, and replace the letters in this word by the corresponding permutations $\sigma_1(F), \ldots, \sigma_n(F)$. 

In this manner the groups $B_n \simeq \Mod(D_n), \Mod(S^2_n)$ act on the set of Hurwitz factorizations, with the action of the former factoring through the latter since the kernel acts trivially.  Moreover, given $F$, the discussion in the previous paragraph shows that the orbit of $\sigma(F)$, under the action of both the braid group $B_n$ and the diagonal action of simultaneous conjugation by elements of $S_d$, is a well-defined invariant of the branched cover $F$. 

The formalism done, the proof of the following result is then very short: 

\begin{theorem}
\label{thm:hurwitz}Fix a given realizable admissible abstract branch data. 
Two branched coverings $F, G$ realizing this data are Hurwitz equivalent if and only $\sigma(F), \sigma(G)$ lie in the same $B_n$-orbit, up to simultaneous conjugation. 
\end{theorem}

\pf By varying $F$ and $G$ within their Hurwitz classes as above, we may assume $V(F)=V(G)=V$ and that the ramification of $F$ and of $G$ above each point $v_i \in V$ is given by the given branch data.   Consider the corresponding unramified coverings 
$F: S^2 \setminus F^{-1}(V) \to S^2\setminus V$ and $G: S^2\setminus G^{-1}(V) \to S^2\setminus V$. If $F, G$ are Hurwitz equivalent via a pair $h, h_1$ with $h: (S^2, V) \to (S^2, V)$, then the unramified covering spaces given by $h \circ F$ and $G$ are isomorphic, and from the above discussion, the Hurwitz factorizations of $F$ and $G$ lie in the same orbit.   Conversely, if the Hurwitz factorizations of $F$ and $G$ differ by the action of a homeomorphism $h$, then the corresponding monodromy representations of the unramified coverings $h\circ F$ and $G$ are isomorphic, and so $h$ lifts to a homeomorphism $h_1: S^2\setminus F^{-1}(V) \to S^2 \setminus G^{-1}(V)$ which extends over the puntures to a homeomorphism $h_1: S^2 \to S^2$ so that the pair $(h, h_1)$ gives a Hurwitz equivalence from $F$ to $G$.  
\qed
\gap

Practically speaking, then, to determine whether two given branched coverings $F, G$ realizing the same branch data are Hurwitz equivalent, we do the following. Choose basepoints $b_F, b_G$ in $S^2\setminus V(F), S^2\setminus V(G)$, respectively. Choose an indexing of $V(F), V(G)$ so that the ramification of $F$ and $G$ above the $ith$ point are isomorphic. Choose a set of rays in the codomains of $F$ and of $G$ running from the basepoint to the critical values. This data determines, up to homeomorphism respecting the labelling of critical values, an identification of each of the codomains of $F$ and $G$ with  the standard codomain and associated set of rays and critical values constructed above.  One then computes the associated Hurwitz factorizations $\sigma(F), \sigma(G)$, which may be read off from a combinatorial presentation of $F$ and of $G$.  The problem of Hurwitz equivalence is then reduced to an (albeit complicated) problem in algebra.

\gap

\subsubsection{Extension to augmented branched coverings.}  Assume that a given abstract augmented branch data $(A \union C \to B, A \union C \to \N)$ is realizable by an augmented branched cover, $f$.  Consider the problem of deciding the augmented Hurwitz equivalence of two such realizations $f, g$. One may attempt to adapt the classical framework to this new setting. However, if $A$ does not consist of all inverse images (counted with multiplicity) of elements of $B$, then a homeomorphism $h$ for which $h\circ f$ and $g$ are covering equivalent need not have a lift $h_1$ which sends $A(f)$ to $A(g)$ as is required in the definition of augmented Hurwitz equivalence.   One may attempt to demonstrate that some subgroup of the braid group will determine an equivalence relation wherein such lifts $h_1$ must send $A(f) \to A(g)$. The difficulty is that such lifts need not be unique, and formulating an appropriate group action is tricky.  In the ``pure'' category this has been accomplished in \cite[\S 3.1]{kmp:kps}. 

\subsection{Functional compositions and decompositions}

A Thurston map $f: (S^2, P) \to (S^2, P)$ may factor as a composition of augmented branched coverings. 
\gap

\noindent{\bf Example \cite[\S 5.1]{bekp}.} The quartic polynomial 
\[ f(z)=2i\left(z^2-\frac{1+i}{2}\right)^2\]
has postcritical set $P:=\{0, 1, -1, \infty\}$. It factors as $f=g\circ s$ with $s(z)=z^2$ and $g(z)=2i(z-\frac{1+i}{2}))^2$. Put $A:=\{0, 1, \infty\}$. Then $f$ factors as a composition of augmented branched covers 
\[ (\IP^1, P) \stackrel{s}{\longrightarrow} (\IP^1, A) \stackrel{g}{\longrightarrow}(\IP^1, P)\]
as is easily verified. The dynamical properties of $f$ are signficiantly influenced by the non-dynamical properties of its compositional factors.  For example, here $\#A=3$, so the pullback map on Teichm\"uller space induced by $f$ is constant. 

If a Thurston map $f$ with fixed critical points factors nontrivially as above, it is equivalent to $z \mapsto z^n$ for some $n$ with $|n|\geq 2$. For if $f=g\circ s$ then each critical point of $g$ must be a critical value of $f$ above which $f$ is totally ramified (else there are critical points collide). Thus $s$ and $g$ each have two critical points and two critical values, with $C(g)=V(s)$, so after conjugation the claim follows.   Thus critically fixed Thurston maps are, with the above exception, all indecomposable from a function-theoretic point of view. 

The preceding example demonstrates the importance of considering non-dynamical objects--- augmented branched coverings and augmented branch data---when considering the classification theory for Thurston maps.  The next example, which arose in conversations with S. Koch and A. Saenz, shows that functional factors of Thurston maps can be augmented branched covers which do not themselves arise from Thurston maps via forgetting identification of domain and range.   

\noindent{\bf Example.}  We identify $\rs$ with $S^2$.  Let $g: (\rs,P(g)) \to (\rs,P(g))$ be an integral Latt\`es example induced by $z + \Lambda \mapsto 3z+\Lambda$ on a complex torus $\C/\Lambda$.  Thus $\deg(g)=9$, and each of the 16 critical points of $g$ has local degree two.  These critical points map surjectively to the postcritical set $P(g)$ of $g$, and each element of $P(g)$ is a repelling fixed-point of $g$.  Pick $p \in P(g)$, and let $h: S^2 \to S^2$ be an orientation-preserving homeomorphism which sends $P$ bijectively onto $g^{-1}(P(g))\setminus\{p\}$.  Finally, put $f=h\circ g$.  The map $f$ is a {\em nearly Euclidean Thurston map} in the sense of \cite{cfpp:netmaps}.  

Here is a description of the critical orbit portrait of $f$. We may identify 
\[ C(f)=\{a_0, \ldots, a_{11}, b_0, b_1, b_2, c\}\]
so that 
\[ a_i \mapsto b_{i\negmedspace\negmedspace \mod3} \mapsto c \mapsto c\]
and thus $P(f)=\{b_0, b_1, b_2, c\}$.  

We now find a factorization of the augmented branched cover $f: (S^2, P(f)) \to (S^2, P(f))$.  Separating domain and range, set $\sfA:=P(f)=\{\sfa_1, \ldots, \sfa_4\}, \sfC:=P(f)=\{\sfc_1, \ldots, \sfc_4\}$.  It is easy to see that $f$ factors as a composition of degree three maps 
\[ (S^2, \sfA) \stackrel{g}{\longrightarrow} (S^2, \sfB)  \stackrel{s}{\longrightarrow} (S^2, \sfC) \] 
where 
\[ \sfa_1, \sfa_2, \sfa_3 \stackrel{g, 1:1}{\longrightarrow} \sfb_4, \;\; \sfa_4 \stackrel{g, 2:1}{\longrightarrow} \sfb_5\]
and 
\[ \sfb_i \stackrel{s, 1:1}{\longrightarrow} \sfc_i, i=1, \ldots, 3, \;\; \sfb_4 \stackrel{s, 2:1}{\longrightarrow} \sfc_4, \;\; \sfb_5 \stackrel{s, 1:1}{\longrightarrow}\sfc_4.\]
Since $\#\sfB=5>4=\#\sfA$, we conclude that the augmented branched covering $g: (S^2, \sfA) \to (S^2, \sfB)$, which is a compositional factor of the Thurston map $f$, cannot itself arise from a Thurston map by forgetting identification of domain and range.  

\subsection{Dynamical compositions and decompositions}    

Suppose $P \subset S^2$ is a finite set. A {\em multicurve} is a collection $\Gamma=\{\gamma_1, \ldots, \gamma_k\}$ of disjoint, pairwise non-homotopic, simple, closed, essential, unoriented, nonperipheral curves up to homotopy in $S^2\setminus P$. If $f: (S^2, P) \to (S^2, P)$ is a Thurston map, a multicurve $\Gamma$ is {\em invariant} if each preimage of an element of $\Gamma$ is either inessential, peripheral, or homotopic in $S^2\setminus P$ to an element of $\Gamma$. 

If $f^{-1}(\Gamma)=\Gamma$  (in the sense that for each $\gamma_i \in \Gamma$, there exists $\gamma_j \in \Gamma$ such that $\gamma_i$ is homotopic in $S^2-P$ to a component of $f^{-1}(\gamma_j)$), i.e. $\Gamma$ is {\em fully invariant}, then by cutting along $\Gamma$ one can decompose, in a dynamically natural way, a Thurston map into simpler pieces, see e.g. \cite{kmp:cds}. In the case $f=F_\GGG$ for some graph $\GGG$, the boundary of any sufficiently small regular neighborhood of an edge will be a fully invariant multicurve, see Lemma \ref{lemma:getmc} below. 

Of course, one could formulate a procedure inverse to that of blowing up an arc, and regard such surgery as a decomposition or reduction of the dynamics. 

\subsection{Wreath recursions}
\label{secn:wreath_recursion}

Suppose $f: (S^2,P) \to (S^2,P)$ is a Thurston map of degree $d$.  Fix a basepoint $b \in S^2-P$ and, for convenience, set now $G:=\pi_1(S^2-P, b)$ and $X:=\{1, \ldots, d\}$.  Fix a bijection $\Lambda: X \to f^{-1}(b)$ and use this to identify $X$ with $f^{-1}(b)$.  For each $x \in X$, choose a path $\lambda_x: [0,1] \to S^2-P$ joining $b$ to $x$.  

This data determines a {\em wreath recursion} of $G$: a homomorphism 
\[ \Phi_f: G \to G^d \rtimes S_d\]
\[ g \mapsto (g_1, \ldots, g_d)\sigma(g).\]
See \cite{nekrashevych:book:selfsimilar}.  Here, $\sigma(g)$ is the monodromy permutation of $X$ induced by $g$, with right action written $i \mapsto i^\sigma$ and 
\[ g_i = \lambda_i*f^{-1}(g)[i]*\overline{\lambda}_{i^\sigma}.\]
(The asterisk stands for concatenation of paths; $\lambda_i$ is traversed first; $f^{-1}(g)[i]$ is the lift of $g$ under $f$ based at $i$; the notation $\overline{\lambda}$ means the reversal of the path $\lambda$.)

Multiplication in $G^d\rtimes S_d$ is computed like this.  First, some conventions: permutations are {\em right} actions.  So $\sigma = (1234)$ means 
\[ (1234) = (1 \mapsto 2 \mapsto 3 \mapsto 4 \mapsto 1).\]
We often write $2=1^\sigma$; equivalently, $2=\sigma(1)$.  Notice that $(123)(35)=(1253) \neq (5123)$ according to our convention; permutations are composed left to right.  With this convention, here is how we multiply in $G^d\rtimes S_d$:  
\[ \left( \genby{g_1, \ldots, g_d}\sigma\right) * \left(\genby{h_1, \ldots, h_d}\tau\right)  = \genby{g_1h_{1^\sigma}, \ldots, g_dh_{d^\sigma}}\sigma\tau.\]
For example:
\[ \genby{1, ab, a, b^{-1}}(142)*\genby{ab, a^{-1}, 1, b}(134) = \genby{1.b, ab.ab, a.1, b^{-1}a^{-1}}(234).\]
This is nicely represented by regarding the elements as decorated permutations, and then concatenating them like braids: 
\[
\xymatrix{ 
1 \ar[rrrd]^1 & 2\ar[dl]_{ab} & 3\ar[d]^<{a}& 4 \ar[dll]^<<<<{b^{-1}} \\
1 \ar[drr]^<<<<<<{ab} & 2 \ar[d]^<<{a^{-1}}& 3 \ar[dr]^{1} & 4 \ar[dlll]^<<<<<<<<<<<<{b}\\
1 & 2 & 3 & 4 
}
\]

The wreath recursion may be conveniently encoded as follows.  For convenience we identify $S^2$ with $\rs$.  If $\#P=n$ we may identify $P=\{p_1, \ldots, p_n\}$ with $\{\exp(2\pi i k/n), k=1, \ldots, n\}$, where $p_k = \exp(2\pi i k/n)$.  For each $k=1, \ldots, n$ we choose as in \S \ref{subsecn:hurwitzfactor} a generator $g_k$ which runs radially from the origin to $p_k$, loops around $p_k$ counterclockwise, and returns to the origin radially.  Then the wreath recursion of $f$ is encoded by the list of $n$ expressions 
\[ \Phi(g_k) = (w_{1k}, \ldots, w_{dk})\sigma(g_k), \;\; k=1, \ldots, n\]
where $w_{ik}$ are words in the generators $g_k$.  If a topological model for $f$ is given, the monodromy induced by $g_k$ and the words $w_{ik}$ may be calculated; see \S \ref{secn:wreath_recursion}.  

Up to a suitable notion of equivalence, the wreath recursion is a complete invariant of the Thurston combinatorial class of $f$.   The following is a restatement of \cite[Theorem 6.5.2]{nekrashevych:book:selfsimilar} into equivalent language.  

\begin{theorem}
\label{thm:wr_determines_f}
Suppose $f: (S^2, P(f)) \to (S^2, P(f))$ and $g: (S^2, P(g)) \to (S^2, P(g))$ are Thurston maps.  Choose basepoints $b_f \in S^2-P(f), b_g \in S^2-P(g)$, and let $G_f, G_g$ be the corresponding fundamental groups of $S^2-P(f), S^2-P(g)$, respectively.  Then $f$ and $g$ are combinatorially equivalent via a homeomorphism $h: (S^2, P(f)) \to (S^2, P(f))$ if and only if the following diagram 
\[
\xymatrix{ 
G_f 
\ar[r]^{\Phi_f} \ar[d]_{h_*} & G_f^d \rtimes S_d 
\ar[d]^{h^d_*\times \id}  \\ 
G_g 
\ar[r]_{\Phi_g} 
& G_g^d \rtimes S_d 
}
\]
commutes up to inner automorphisms of $G_g^d\rtimes S_d$ induced by conjugation by elements of $S_d$. 
\end{theorem} 

L. Bartholdi \cite{bartholdi:program:img} has implemented a robust general version of the iterative algorithm on Teichm\"uller space used in the proof of Thurston's characterization theorem.  Given a wreath recursion for a Thurston map, it is designed to return either a numerical approximation for the corresponding rational map, or an algebraic description of an obstruction.

\subsection{Topological polynomials}

A Thurston map $f$ is a {\em topological polynomial} if has a branch point $c$ with $f(c)=c$ and $\deg(f,c)=\deg(f)$. Analogously, an abstract critical portrait is of polynomial type if it has a fixed-point of weight equal to its degree; a branch datum is of polynomial type if it has a point in its domain of weight equal to its degree.  The combinatorial theory of polynomials is much simpler than that for rational functions, due mainly to the existence of good normal forms (Hubbard trees; critical (external ray) portraits; invariant spiders; kneading automata; etc.) for various invariants; see \cite{nekrashevych:combinatorics}.

\section{Topological realizability conditions}

It is natural to consider branched coverings $S^2 \to S^2$ without regard to dynamics, as well as to consider abstract directed graphs representing potential branch data, augmented branch data, and critical portraits. In this section, we consider the problem of characterizing the images of the natural functors between the categories given in \S \ref{subsecn:summary}.  

\subsection{$\mathbf{TM}\to\mathbf{BC}$} A branched covering $F: S^2 \to S^2$ determines a map of pairs $F:(S^2, C(F)) \to (S^2, V(F))$ where $\#V(F) \leq \#C(F)$, so there exists a (many, really) homeomorphism $H$ yielding a map of pairs $H: (S^2, V(F))\to (S^2, C(F))$. The composition $f:=H\circ F$ will be a Thurston map and so every branched covering arises from some Thurston map.  So $\TM \to \BC$ is (on the level of objects) surjective.  A Hurwitz equivalence between the branched covers corresponding to two Thurston maps can not always be lifted to a Thurston equivalence, however. 

\subsection{$\mathbf{BC}\to\mathbf{BD}$ }
\label{subsecn:bc_to_bd}
 Fix an integer $d \geq 2$.  
While necessary, admissibility is not sufficient for realizability of branch data by a branched covering. 
\gap

\noindent{\bf Example:} The branch data of degree 4 given by the partitions 
\[ \{[3,1], [2,2], [2,2] \}\]
is not realized by any cubic branched covering. To see this, assume the contrary, and consider the associated Hurwitz factorization.  In this case, this is an ordered triple of permutations $(\alpha, \beta, \gamma)$ in $S_4$ generating a transitive subgroup and with $\alpha \beta \gamma = 1$. We may assume $\alpha$ is the three-cycle $(123)$ and $\beta$ is a product of two disjoint transpositions. By simultaneously conjugating by a power of $\alpha$ we may assume $\beta= (14)(23)$. But then $\gamma = \beta^{-1}\cdot \alpha^{-1}=(143)$ is not a product of transpositions. 
\gap

Indeed, a complete characterization is still unknown; for recent work see \cite{pervova:petronio:existenceII} and the references therein. 
To give a flavor of the known results, we mention the following. 
\bi
\item Any admissible branch datum in degrees 2, 3, 5, 7 is realizable \cite{eks}
\item In each non-prime degree there exist non-realizable admissible branch data \cite{eks}. 
\item Any polynomial branch datum is realizable \cite{khovanskii:zdravkovska:realizable} 

\ib
In what follows, we shall always assume tacitly that branch data is required to be admissible in the sense of \S \ref{subsecn:branchdata}.

\subsection{$\portrait \to \BD, \TM \to \portrait$} The preceding two paragraphs show that if some admissible abstract branch data $(C \to V, C \stackrel{\deg}{\longrightarrow} \N)$ is realizable, then it arises via restriction and forgetting identification of domain and range from some critical portrait of a Thurston map; this is \cite[Prop. 2.9]{kmp:census}. They also show that if the abstract admissible branch data associated to some portrait is realizable, then this portrait arises from a Thurston map.  

\subsection{$\TM \to \ABC, \portrait \to \ABD$} While natural in some nondynamical contexts, we do not pursue here the formulation of abstract augmented branched covers, abstract augmented branch data, and the characterization of such objects that arise from Thurston maps.  

\section{Holomorphic realizability} 

Fix an identification $S^2 = \IP^1=\rs$. 

Letting the prefix ``$\mathbf{H}$'' stand for holomorphic, we get three corresponding subcategories $\mathbf{H}\TM, \mathbf{H}\portrait$, $\mathbf{H}{\BC}$ where both the maps defining objects and those defining isomorphisms are in addition holomorphic. The objects of $\mathbf{H}\TM$ are thus postcritically finite rational maps and the morphisms elements of $\Aut(\IP^1)$ giving conjugacies.  The objects of $\mathbf{H}\BC$ are rational maps; there is a morphism $f \to g$ if $f$ and $g$ differ by pre- and post-composition by (typically) distinct elements of $\Aut(\IP^1)$. 

The lifting of complex structures under coverings, the removable singularities theorem, and the uniformization theorem  imply that any branched covering is, up to pre-composition with homeomorphisms, a holomorphic map; this is sometimes attributed to R. Thom. It follows that a topologically realizable branch datum is isomorphic to the branch datum of an element of $\mathbf{H}\BC$. 

Thurston's fundamental characterization and rigidity theorem \cite{DH1} characterizes elements of $\mathbf{H}\TM$ among elements of $\TM$, and asserts that apart from the ubiquitous family of flexible Latt\`es counterexamples, if $f, g \in \mathbf{H}\TM$ are isomorphic as elements of $\TM$, then they are isomorphic as elements of $\mathbf{H}\TM$. 

An unpublished theorem of Berstein-Levy  \cite{hubbard:teichvol2}  gives a sufficient condition for a topological polynomial to be combinatorially equivalent to a postcritically finite complex polynomial.  

\begin{theorem}[Berstein-Levy] 
\label{thm:berstein-levy}
Suppose $f$ is a topological polynomial. If each cycle of its critical  orbit portrait contains a critical point, then $f$ is combinatorially equivalent to a complex polynomial. 
\end{theorem}

\begin{corollary}
\label{cor:twists_realizable}
Suppose $f$ is a topological polynomial. If each cycle of its critical  orbit portrait contains a critical point, then for each homeomorphism $h: (S^2, P(f)) \to (S^2, P(f))$ fixing the point at infinity, the map $h \circ f$ is combinatorially equivalent to a complex polynomial. 
\end{corollary}

Kelsey \cite{kelsey:schemes} provides a partial converse.  If the critical orbit portrait of a topological polynomial has a cycle not containing a critical point, then under very general further conditions on the portrait, there exists an obstructed topological polynomial with this portrait; there may (and often do) exist complex polynomials realizing this portrait as well.

Not every critical portrait arises from a complex polynomial. 
\gap

\noindent{\bf Example.} Suppose the cubic critical portrait consists of four fixed-points each of local degree two. Then this portrait does not arise from any cubic rational map: the Holomorphic Fixed-Point  Formula \cite[Theorem 12.4]{milnor:dynamics} asserts that if all fixed-points are simple then 
\[ \sum_{z \in \Fix(f)} \frac{1}{1-f'(z)} = 1.\]
A postcritically finite rational map has all fixed-points simple (multiple fixed-points are parabolic fixed-points; their immediate basins necessarily contain critical points with infinite forward orbit; see {\em op. cit.}, so a cubic rational map cannot have four fixed critical points.  
\gap

A general characterization of holomorphically realizable critical orbit portraits seems at present far out of reach. Even the enumeration of critical orbit portraits is challenging; cf. \cite{kmp:census} where it is shown by brute force exhaustion that there are 134 admissible distinct cubic critical orbit portraits which have all cycles attractors, are non-polynomial type, and have four or fewer postcritical points. 

\section{Uniqueness of realizations} 

Here, we consider the fibers of the natural functors in \S \ref{subsecn:summary}. That is, we discuss when two topological invariants have the same set-theoretic invariants.

\subsection{$\TM \to \portrait$} In \cite[Prop. 2.12 ]{kmp:census} there is an explicit example given of a critical orbit portrait which arises from infinitely many distinct combinatorial classes. The idea is simple: one finds first a Thurston map $f_1$ and a subset $\Sigma$  with  $\Sigma \hookrightarrow S^2 \setminus P(f)$ injective on fundamental group and satisfying $f_1|_\Sigma = \id_\Sigma$.  Postcomposing $f_1$ with homeomorphisms $h$  supported on $\Sigma$ and representing distinct conjugacy classes of mapping classes in $\PMod(S^2, P(f_1))$ yields the desired collection of Thurston maps $f$. 

Another explicit class of such examples is given in \cite[\S Prop. 6.10]{bartholdi:nekrashevych:twisted}. 

And, unsurprisingly, there are other families of examples coming from the exceptional class of Thurston maps with Euclidean orbifold. 

In constrast, Thurston's fundamental rigidity theorem implies that apart from the flexible Latt\`es examples, there are up to holomorphic conjugacy only finitely many postcritically finite rational maps realizing a given critical orbit portrait; see \cite[Cor. 3.7]{kmp:census}. 

\subsection{$\BC \to \BD$} For ``generic'' branch data (all critical points simple, i.e. of local degree two, and mapping to distinct critical values) a classical result of Clebsch and L\"uroth asserts that there is a unique Hurwitz class realizing the data; see \cite[Theorem 3.4]{berstein:edmonds:construction} for a detailed proof in modern language.  Indeed, a concrete model for a representative can be given by starting with the identity map and ``blowing up'' $d-1$ disjoint arcs; see \S \ref{subsecn:blow}. 

In principle, Hurwitz equivalence is algorithmically checkable. Indeed, given Hurwitz data for branched covers $f$ and $g$, starting with the data for $f$ one may randomly apply the elementary moves corresponding to generators of the mapping class group and check to see if the result is equivalent to the data for $g$. This can be done very quickly, and if $f$ and $g$ are equivalent it seems likely that one may find such an equivalence in a reasonable amount of time if the degree is not too large.   

It is relatively easy to find examples of Hurwitz inequivalent covers with the same branch data. 
\gap

\noindent{\bf Example 1:} Consider the set of so-called {\em clean Belyi polynomials} $f: \C \to \C$ defined by the conditions that $V(f) = \{0,1\}$ and $\deg(f,c)=2$ for each $c \in f^{-1}(1)$. Such a polynomial has a {\em dessin}, the bipartite planar tree whose underlying set is $f^{-1}[0,1]$ and whose vertices are $f^{-1}(\{0,1\})$, with preimages of $0$ being black and preimages of $1$ being white. It is easy to prove, using triviality of the pure mapping class group of the thrice-marked sphere,  that if $f, g$ are two such polynomials each having at least one finite critical point of local degree at least $3$, then $f, g$ are Hurwitz equivalent if and only if the covering spaces of $\C\setminus\{0,1\}$ they induce are isomorphic. This, in turn, holds if and only if their dessins are isomorphic as planar embedded bipartite trees. So to find examples it suffices to draw two such dessins which are isomorphic as abstract, but not as planar, trees.  An explicit example may be found in \cite[\S 4]{kmp:dessins}.  The point is that the branch data records only the cycle structure of the monodromy of peripheral generators, while the conjugacy class of the subgroup of the symmetric group determined by monodromy is an invariant of the Hurwitz class. 
\gap 

It is  more difficult to give, say, primitive such pairs of examples with more branch values. 
\gap

\noindent{\bf Example 2:} The following is due to Jones and Zvonkin \cite{jones:zvonkin:cacti}.  The polynomial branch data of degree 7 corresponding to the partitions
\[ \{ [7], [2,2,1,1,1], [2,2,1,1,1] \}\]
is topologically realizable by precisely 4 distinct Hurwitz classes. 
\gap

In addition to the case of generic branch data, there are several partial results giving conditions on the branch data for there to exist a unique Hurwitz class.   Khovanskii and Zdravkovska  \cite[\S 4] {khovanskii:zdravkovska:realizable} use analytical methods to show this is the case for topological polynomials if the sum of the local degrees at critical points is not too large compared to the degree.  Relevant for the case of Thurston maps with fixed critical points is the following recent result of Liu-Osserman; they give a proof but assert that this result is well-known. 

\begin{theorem}
\label{thm:purecycle}
Suppose a given admissible branch datum has the property that each associated partition has the form $[k, 1, 1, \ldots, 1]$. Then any two branched coverings realizing this datum are Hurwitz equivalent. 
\end{theorem}

That such branch data is realizable seems to be well-known; our Theorem \ref{theorem:ptomg} gives a combinatorial proof.   

\subsection{$\ABC \to \ABD$} It is easy to find examples of augmented branched covers which are Hurwitz equivalent, but not augmented Hurwitz equivalent; the following is taken verbatim from \cite[\S 2.9]{koch:criticallyfinite}. We warn the reader that the term ``Hurwitz equivalent'' used in \cite{koch:criticallyfinite} is not the same as that used here; Koch's notion is here (see below) called ``purely augmented Hurwitz equivalent''.  
\gap

\noindent{\bf Example.} Let $F_\alpha(t)=(1-(1-\alpha)t)^d$ where $\alpha \neq 1$ is a $d$th root of unity. Each $F_\alpha$ is a unicritical polynomial of degree $d$, with critical point $1/(1-\alpha)$, such that 
\[ F_\alpha(\frac{1}{1-\alpha})=0, \ F_\alpha(0)=1, \ \mbox{ and } F_\alpha(1)=1.\]
Define the sets $A:=\{0,1,\infty\}$ and $B:=\{0,1,\infty\}$. Now suppose $\alpha, \beta$ are distinct $d$th roots of unity. Suppressing mention of local degrees when they are equal to $1$, the corresponding augmented branch data for $F_\alpha$ and $F_\beta$ are given respectively by 
\[ \{0 \mapsto 1, \; 1 \mapsto 1, \; \frac{1}{1-\alpha} \stackrel{d}{\mapsto} 0, \; \infty \stackrel{d}{\mapsto}\infty\}\]
and
\[ \{ 0 \mapsto 1, \; 1 \mapsto 1, \; \frac{1}{1-\beta} \stackrel{d}{\mapsto} 0, \; \infty \stackrel{d}{\mapsto}\infty.\}\]
Clearly, the isomorphism type of this augmented branch data depends only on $d$, and the automorphism group of this augmented branch data (defined in the obvious way) is generated by two commuting involutions: one which swaps $0\leftrightarrow 1$ in $A$ leaving $B$ fixed pointwise, while the other interchanges $\frac{1}{1-\alpha} \to 0$ with $\infty \to \infty$.

For varying choices of $\alpha$, the maps 
\[ F_\alpha: (\P^1, A) \to (\P^1, B)\]
are all ordinary Hurwitz equivalent (in the sense of this paper) to each other, since the underlying branched coverings have the same degree and are ramified only over two points.  However, if $\alpha \neq \beta$ then $F_\alpha$ and $F_\beta$ will not in general be augmented Hurwitz equivalent  in the sense of \S \ref{subsecn:augbc}.  

To see this, suppose  the pair $(h_0, h_1)$ gives such an equivalence.  We must have $h_0(1)=1$. The (non-pure) mapping class group $\Mod(\P^1, B)$ is represented by the six elements of $\Aut(\P^1)$ preserving the set $\{0,1,\infty\}$, which includes $z \mapsto 1/z$. This element lifts under any map $F_\beta$, so without loss of generality we may assume further that $h_0(0)=0$ and $h_0(\infty) = \infty$, so that $h_0(z)=\id$. Since $h_0, F_\alpha, F_\beta$ are holomorphic, so is $h_1$.  From the definition of augmented Hurwitz equivalent, we must have $h_1(\infty)=\infty$, so $h_1$ is affine, and we must also have $h_1(\{0,1\})=\{0,1\}$, so either $h_1(z)=z$ or $h_1(z)=1-z$. But we must also have $h_1(\frac{1}{1-\alpha})=\frac{1}{1-\beta}$, so $\alpha \neq \beta$ implies $\frac{1}{1-\alpha}+\frac{1}{1-\beta}=1$. We conclude that $F_\alpha, F_\beta$ are augmented Hurwitz equivalent only if $\beta=\overline{\alpha}$.   A straightforward calculation shows that in this latter case the maps $F_\alpha, F_\beta$ are in fact augmented Hurwitz equivalent.

We remark here that the arguments of \cite[\S 2.9]{koch:criticallyfinite} show that $F_\alpha, F_\beta$ are never purely augmented Hurwitz equivalent if $\alpha \neq \beta$, and so these examples show that there is indeed a  distinction between the two notions. 
\gap

\noindent{\bf Remark:} In many applications, including dynamical ones, in order to avoid singularities caused by configurations with symmetries, a stronger notion of Hurwitz equivalence is used which requires in essence that equivalences preserve a labelling of points. 
This is the point of view taken in \cite{koch:criticallyfinite}. 

\section{Proof of Theorem \ref{theorem:aretwists} }

Suppose $F, G$ are two Thurston maps whose critical points are all fixed, and suppose the number $n$ and multiplicities of their critical points coincide, so that both have isomorphic branch data.  Let $V=\{\exp(2\pi i k/n), k=1, \ldots, n\}$ be as in \S \ref{subsecn:hurwitzfactor}. 

By varying $F, G$ within their respective combinatorial classes, we may assume $C(F)=C(G)=V(F)=V(G)=P(F)=P(G)=P:=V$, and that the local degrees of $F$ and $G$ each each $c \in C$ coincide. 
By \cite[Proposition 1.1(a)]{liu:osserman:irreducible}, the Hurwitz factorizations associated to $F$ and $G$ lie in the same orbit of the pure braid group. It follows that there exist orientation-preserving homeomorphisms $h_0, h_1: (S^2, P) \to (S^2, P)$ such that $ h_0\circ F = G \circ h_1$; since the braid is pure, we have $h_0|P=\id_P$.  The functional equation and the fact that over each $v \in V$ the fiber of $G$ contains a unique critical point then implies that $h_1|P=\id_P$ as well and so both $h_0, h_1$ lie in $\PMod(S^2, P)$. 

Set $h:=h_1^{-1}\circ h_0$. Then 
\[ h\circ F = h_1^{-1}\circ G \circ h_1\]
which shows that $G$ is combinatorially equivalent to a twist of $F$. 
\qed

\section{From planar multigraphs to Thurston maps} 

\subsection{Planar multigraphs} Here, a {\em multigraph} $\GGG$ is a finite graph $(V,E)$ with a nonempty set of vertices, without loops (edges from a vertex to itself), without isolated vertices (i.e. each vertex is incident to some edge), and with multiple (``parallel'') edges joining distinct vertices permitted.  The underlying {\em simple graph} $S\GGG$ is the graph obtained by removing any redundant edges.  We encode $\GGG$ by weighting each edge of $S\GGG$ according to its multiplicity. $\GGG$ is {\em connected} if $S\GGG$  is connected. 

A multigraph is $\GGG$ {\em planar} if its underlying simple graph $S\GGG$ can be drawn in the plane without crossing of edges. We will always consider planar multigraphs to be so embedded. Two planar multigraphs are here termed {\em isomorphic} if there is an orientation-preserving homeomorphism sending one to the other and respecting the multiplicity of edges. 

The case of $\#V=2$ will be elementary and largely uninteresting; we assume throughout that $\#V \geq 3$ to avoid repeated mention of cumbersome special cases.  In this case, each face of $S\GGG$ has at least three vertices on its boundary.  In contrast, $\GGG$ may have bigons: faces bounded by a pair of parallel edges.

\subsection{Arcs} Suppose $P \subset S^2$ is a finite set. An {\em arc} in $(S^2, P)$ is an embedded arc in $S^2$ with endpoints in $P$ and which otherwise avoids $P$; two such arcs are {\em homotopic} if they can be continuously deformed one to the other through such arcs. If $F: (S^2, A) \to (S^2, B)$ is an augmented branched covering and $\beta$ is an arc in $(S^2, B)$ then (abusing terminlogy) by a {\em preimage} $\tilde{\beta}$ of $F^{-1}(\beta)$ we will mean the closure of a component of $F^{-1}(\beta \setminus \mbox{endpoints}(\beta))$; we say $\tilde{\beta}$ is {\em essential} if it is an arc in $(S^2, A)$.  If $A=B=P$ then an arc $\alpha$ is {\em invariant} under pullback by $F$ if it has a preimage homotopic to itself in $(S^2, P)$.

\subsection{Blowing up an arc and proof of Theorem 3} 
\label{subsecn:blow}

For details, see \cite{kmp-tan:blow}. The process of {\em blowing up an arc} is a surgery procedure for modifying a given Thurston map to produce another Thurston map of higher degree. It is based on the following idea.

Suppose $F: (S^2, P) \to (S^2, P)$ is an augmented branched covering; we allow the case $\deg(F)=1$. Suppose $\alpha$ is an arc in $(S^2, F^{-1}(P))$, and suppose $F|\alpha$ is a homeomorphism onto its image.  To modify $F$ by blowing up along $\alpha$ to obtain a new map $G$, we cut the domain along $\alpha$, open up the slit, map the complement via $F$, and map the disk homeomorphically to the complement of $F(\alpha)$ in the obvious way; see \cite[Figure 3]{kmp-tan:blow} for details.  The combinatorial class of $G$ depends only on the pair $(F,\alpha)$ in the obvious way; an equivalence $F \to F'$ sending $\alpha \to \alpha'$ yields an equivalence $G \to G'$ between the corresponding modified maps, so the process is natural.

One can blow up the same arc multiple times as well, by mapping the complement in the obvious many-to-one fashion.  Up to homotopy, the result is equivalent to blowing up each member in a family of parallel homotopic arcs joining the same endpoints; this point of view is sometimes convenient. Thus, if $\GGG$ is a planar multigraph, there is an associated Thurston map $F_\GGG$ of degree equal to one plus the number (counted with multiplicity) of  edges of $\GGG$.  Clearly, a planar isomorphism $\GGG_1 \to \GGG_2$ induces a combinatorial equivalence $F_{\GGG_1} \to F_{\GGG_2}$.  So we have a procedure that starts with a planar multigraph $\GGG$ and builds a Thurston map $F=F_\GGG$ for which $\Crit(F) \subset \Fix(F)$. 
\gap 

Suppose $F=F_\GGG$ is obtained by blowing up edges in a graph $\GGG$; let $P=P(F)$.  Suppose vertices $v, w$ are joined by an edge $e$ of $S\GGG$ which is blown-up $m(e)$ times; we regard $e$ as an arc in $(S^2, P)$.  By construction, the arc $e$ has $1+m(e)$ components of its preimage that are arcs homotopic to $e$ in $(S^2, P)$. 
Since $m(e) \geq 1$, from \cite[Cor. 3]{kmp-tan:blow} we have that up to homotopy no Thurston obstructions can intersect $e$; from \cite[Cor. 4]{kmp-tan:blow} we have that likewise no invariant arc meets $e$. 
The former observation implies the following, which is a special case of \cite[\S 5.2]{kmp-tan:blow}. 
\begin{theorem}
\label{thm:blowupgraph} 
$F_G$ is combinatorially equivalent to a rational map if and only if $G$ is connected. 
\end{theorem}
The latter of the above observtions implies Theorem \ref{thm:injective} from the introduction, asserting that distinct graphs produce maps which are distinct, up to equivalence.
\gap

\noindent{\bf Proof of Theorem \ref{thm:injective}.}  Suppose $h_0, h_1: (S^2, P_1) \to (S^2, P_2)$ give a combinatorial equivalence between $F_1=F_{\GGG_1}$ and $F_2=F_{\GGG_2}$.  We claim $h:=h_0$ induces a planar isomorphism $h: \GGG_1 \to \GGG_2$.  Consider an edge $e_1$ of $\GGG_1$ having $k\geq 2$ preimages homotopic to itself under $F_1$. Then its image $h(e_1)$ has $k\geq 2$ preimages homotopic to itself under $F_2$. It follows from \cite[Cor. 4]{kmp-tan:blow} that up to homotopy $h(e_1)$ cannot meet an edge of $\GGG_2$. So either $h(e_1)=e_2$ for some edge $e_2$ of $\GGG_2$ or $h(e_1)$ is an arc in a face $D$ of $\GGG_2$ containing at least four vertices on its boundary. The latter case cannot occur: let $\delta$ be an arc in $D$ which joins two nonincident vertices on $\bdry D$ and which intersects $h(e_1)$. Then $\delta$ is invariant under $F_2$ and meets $h(e_1)$ which has more than one preimage under $F_2$ homotopic to itself, which is impossible again by \cite[Cor. 4]{kmp-tan:blow}. Thus $h$ sends edges of $\GGG_1$ to edges of $\GGG_2$; clearly, if $e$ is blown up $m$ times for $F_1$ then $h(e)$ is blown up $m$ times for $F_2$, and the theorem is proved. 
\qed

Invariant multicurves are natural invariant objects associated to a Thurston map.  The next Lemma shows that for maps $f_\GGG$ as above, there are many invariant multicurves with special properties.  

\begin{lemma}
\label{lemma:getmc}Suppose $\GGG$ is a multigraph and $f=f_\GGG$ is obtained by blowing up each edge $e$ of $\GGG$ some number $m(e)$ times. 
Suppose $e$ is an edge of $\GGG$. Let $\gamma$ be the boundary of a small regular neighborhood of $e$, so that $\gamma$ separates the endpoints $v, w$ of $e$ from all other points of $P(f)$. Then $\Gamma:=\{\gamma\}$ is an invariant multicurve.  The curve $\gamma$  has a single essential nonperipheral preimage homotopic to itself, mapping by degree 
\[ \sum_{\stackrel{e' \neq e,}{\tiny \{v,w\} \intersect e' \neq \emptyset}} (1+m(e')),\]
and all other preimages are inessential.  In particular, $\Gamma:=\{\gamma\}$ is a completely invariant multi curve. 
\end{lemma}

\newpage
\section{Proof of Theorem \ref{thm:graphtheory}}

We will refer to the underlying simple graph of $\GGG $ by $S\GGG $, where $V(S\GGG ) = V(\GGG )$ and $v, w \in V(S\GGG )$ are connected by exactly one edge if $v, w$ are connected by at least one edge in $\GGG $. We denote the degree of $v \in v(\GGG )$ in $\GGG $ by $\mathrm{deg}_\GGG (v)$, and the degree of $v$ in $S\GGG $ by $\mathrm{deg}_{S\GGG }(v)$. We consider a face of a simple planar graph $S\GGG $ to be a connected region of the plane that is bounded by edges of $S\GGG $. \\

Suppose $d \geq 3$ is an integer. An {\em admissible partition} of the integer $2d-2$ is a partition $k_1+ \ldots + k_n=2d-2$ where $n \leq d$ and $1 \leq k_i \leq d-1$ for each $i=1, \ldots, n$; the order of the terms is irrelevant, and by convention we arrange the terms so $k_1 \geq \ldots \geq k_n$. Any partition may be encoded by its {\em Young diagram}, an array of congruent squares where each row corresponds to a term in the partition, with the $i$th row containing $k_i$ squares.  This point of view will help to organize our arguments. In this language, Theorem \ref{thm:graphtheory} can be stated as follows. 

\begin{theorem}
\label{theorem:ptomg}
Given any Young diagram associated with an admissible partition of $(n, d)$, we may find a corresponding connected planar multigraph $\GGG $ such that each Young diagram row corresponds to a unique vertex of $\GGG $ and the number of boxes in a row is equal to the degree of the associated vertex. 
\end{theorem}

Note that the restriction $n \leq d$ is necessary. For instance, if $d = 3$, $n = 4$, then we have $2\cdot(3 - 2) = 4$ boxes and 4 rows, so each row must have exactly one square. But clearly there is no way to obtain a connected planar multigraph with 4 vertices each of degree 1. \\

\begin{proof}
By induction on $d$. \\

Note that, given a Young diagram associated with an admissible partition of $(2, d)$, we may find a corresponding connected planar multigraph $\GGG $ that consists of vertices $v_1$ and $v_2$ connected by $d-1$ edges. For the remainder of the proof, we assume $n \geq 3$. \\ 

\noindent \underline{Base Case}:

Since $n, d \geq 3$ and $n \leq d$, we first must show that this holds for $n = d$. We do so by induction on $n$. 

If $n = d = 3$, there is only one possible Young diagram, as below. We see that the only possible corresponding multigraph is a path of length $2$, which we note is connected and planar. \\ \\

Assume that we may find a connected planar multigraph $\GGG $ for any Young diagram associated with an admissible partition of $(k-1, k-1)$. Since $|V(\GGG )| > 2$ and $\GGG $ is connected, we may find $v_1$ and $v_2 \in V(\GGG )$ two adjacent vertices connected by edge $e \in E(\GGG )$. 

We now must show that we may find a connected planar multigraph $\GGG '$ for any Young diagram associated with an admissible partition of $(k, k)$. We note that we may construct any of the Young diagrams associated with an admissible partition of $(k, k)$ (with $2k-2$ boxes) from one of the Young diagrams associated with an admissible partition of $(k-1, k-1)$ (with $2k - 4$ boxes) by adding one row and two new boxes. Let the vertex corresponding to the added row be $v' \in V(\GGG ')$. We consider two cases. \\ 

\noindent \underline{Case I}: 
We add one new row with two boxes. 

In this case, we may construct $\GGG '$ from $\GGG $ by deleting $e \in E(\GGG )$ and adding $e_1, e_2 \in E(\GGG ')$ such that $e_1$ connects $v_1$ and $v'$, and $e_2$ connects $v_2$ and $v'$. Note that the degree of $v_1$ and $v_2$ has not changed, and $v'$ has degree $2$ in $\GGG '$. Furthermore, since $\GGG $ was connected and planar by our inductive assumption, $\GGG '$ is clearly connected and planar. \\ 

\noindent \underline{Case II}: 
We add one new row consisting of one box and also add one other box to some existing row (say the row corresponding to $v_1$). 

In this case, we may construct $\GGG '$ from $\GGG $ by adding $v'$ as a leaf of $v_1$. Note that we have increased the degree of $v_1$ by $1$, and $v'$ has degree $1$ in $\GGG '$. Again, since $\GGG $ was connected, $\GGG '$ is clearly connected. Furthermore, since we may add a leaf to any vertex of a planar multigraph and the resulting multigraph will also be planar, $\GGG $ planar implies that $\GGG '$ is planar. \\

Thus, given $n = d$, we may find a connected planar multigraph for any Young diagram associated with an admissible partition of $(n,d)$. \hfill $\square$ \\

Now, consider $d > n$. For the purposes of induction on $d$, we assume that we may find a connected planar multigraph $\GGG $ for any Young diagram associated with an admissible partition of $(n, d)$. We now must show that we may find such a connected planar multigraph $\GGG '$ for any Young diagram associated with an admissible partition of $(n, d+1)$. We note that we may construct all of the Young diagrams associated with an admissible partition of $(n, d+1)$ from the Young diagrams associated with an admissible partition of $(n, d)$ by adding two new boxes. We proceed by two cases. \\ 

\noindent \underline{Case I}: We add one new box to each of two distinct rows of a Young diagram associated with an admissible partition of $(n, d)$. Recall that, by assumption, $n \geq 3$. We note that this corresponds to increasing the degree of two vertices of $\GGG $ by 1. 

\begin{lemma}
Let $\GGG $ be a connected planar multigraph with $|V(\GGG )| \geq 3$. Given $v_1, v_2 \in V(\GGG )$ adjacent and $\mathrm{deg}_{S\GGG }(v_1) \geq 2$, we may find a connected planar multigraph $\GGG '$ with $V(\GGG ) = V(\GGG ')$ such that 

1. $\mathrm{deg}_{\GGG '}(v) = \mathrm{deg}_\GGG (v)$ for all $v \in V(\GGG )$, $v \neq v_1, v_2$ 

2. $\mathrm{deg}_{\GGG '}(v_1) = -1 + \mathrm{deg}_\GGG (v_1)$

3. $\mathrm{deg}_{\GGG '}(v_2) = 1 + \mathrm{deg}_\GGG (v_2)$. 
\end{lemma}

\begin{proof}
$\GGG $ is a planar multigraph with $|V(\GGG )| \geq 3$, so since $S\GGG $ contains no multiple edges or loops, the boundary of each face of $S\GGG $ contains at least three distinct vertices. Since $v_1$ and $v_2$ are adjacent, they lie on the boundary of some face $F$ of $S\GGG $. Furthermore, since $\mathrm{deg}_{S\GGG }(v_1) \geq 2$, there exists $w \in V(\GGG )$ adjacent to $v_1$, $w \neq v_2$, such that $w$ also lies on $\partial F$. We refer to $w$ as a {\em buffer vertex} of $v_1, v_2$.

Let $\GGG '$ be the graph obtained from $\GGG $ by deleting edge $v_1 w$ and adding edge $w v_2$. By construction, Conditions 1 - 3 hold. Since $w$ and $v_2$ are both on $\partial F$, we may add edge $w v_2$ without crossing any other edges, so $\GGG '$ is planar. Since $\GGG $ is connected, $\GGG '$ is connected by construction. 
\end{proof}

\noindent The proof of Case I then follows from the following proposition: 

\begin{proposition}
Let $\GGG $ be a connected planar multigraph with $|V(\GGG )| \geq 3$. Given arbitrary $s, t \in V(\GGG )$, we may find a connected planar multigraph $\GGG '$ with $V(\GGG ) = V(\GGG ')$ such that 

1. $\mathrm{deg}_{\GGG '}(v)  = \mathrm{deg}_\GGG (v)$ for all $v \in V(\GGG )$, $v \neq s, t$ 

2. $\mathrm{deg}_{\GGG '}(s) = 1 + \mathrm{deg}_\GGG (s)$

3. $\mathrm{deg}_{\GGG '}(t) = 1 + \mathrm{deg}_\GGG (t)$. 
\end{proposition}

\begin{proof}
If $s$ and $t$ are adjacent in $\GGG $, we may construct $\GGG '$ by simply adding another edge between them. Assume $s, t$ not adjacent in $\GGG $. 

Since $\GGG $ is connected, we may find $P_1 = s, a_1, a_2, \ldots, a_m, t$ a path of minimal length between $s$ and $t$. Add edge $s a_1$, increasing the degrees of $s$ and $a_1$ each by 1. Call the resulting graph $\HHH$, and note that $\HHH$ is a connected planar multigraph with $V(\GGG ) = V(\HHH)$. 

Since $a_1$ is adjacent to $s$ and $a_2$, $\mathrm{deg}_{S\HHH}(a_1) \geq 2$. Furthermore, since $P_1$ is a path of minimal length, $a_1$ is not adjacent to any other vertices in $P_1$. Therefore, there exists $w \in V(H)$ adjacent to $a_1$ such that $w \not\in V(P_1) - \{s\}$ and $w, a_1, a_2$ lie on the boundary of some face $F$ of $SH$. We apply Lemma 1 with $w$ as our buffer vertex of $a_1, a_2$ to obtain connected planar multigraph $\HHH'$ with $V(\HHH) = V(\HHH')$ such that 

1. $\mathrm{deg}_{\HHH'}(v)  = \mathrm{deg}_\HHH(v)$ for all $v \in V(\HHH)$, $v \neq a_1, a_2$ 

2. $\mathrm{deg}_{\HHH'}(a_1) = -1 + \mathrm{deg}_{\HHH}(a_1)$

3. $\mathrm{deg}_{\HHH'}(a_2) = 1 + \mathrm{deg}_{\HHH}(a_2)$. 

\noindent Note that we now have $\mathrm{deg}_{\HHH'}(a_1) = \mathrm{deg}_{\GGG }(a_1)$. So, if $a_2 = t$, we are done. Otherwise, since $\HHH'$ is connected, we may find $P_2 = a_2, b_1, \ldots, b_n, t$ a path of minimal length from $a_2$ to $t$. Note that since the path $a_1, a_2, a_3, \ldots, a_m, t \in \HHH'$, the length of $P_2$ is strictly less than the length of $P_1$. We redefine $a_1 := a_2$, $a_2 := b_1$ and repeat the argument. Since the length of our minimal path strictly decreases with each iteration, we will eventually obtain $a_2 = t$, thus proving our claim. 
\end{proof}

\noindent \underline{Case II}: We add two new boxes to one row of a Young diagram associated with an admissible partition of $(n, d)$. Following the given constraints, this case only pertains to rows of size $\leq d -2$. 

Suppose that our row corresponds to $v_1 \in V(\GGG )$. Since $n, d \geq 3$ and $\GGG $ is planar, we may find $v_2, v_3 \in V(\GGG )$ such that $v_2$ and $v_1$ are connected by edge $e_{12}$, $v_2$ and $v_3$ are connected by edge $e_{23}$, and $v_1, v_2, v_3$ lie along the boundary of a face $F$ of $S\GGG $. Now, to increase the degree of $v_1$ by 2 without changing the degrees of any of the other vertices of $\GGG $, we may simply delete $e_{23}$ and add an edge from $v_3$ to $v_1$ and another edge from $v_2$ to $v_1$ to form $\GGG '$. Note that since $\GGG $ is connected, $\GGG '$ is connected. Furthermore, as previously argued, we may connect $v_1$ and $v_3$ without violating planarity, since they are both on $\partial F$. \\

Thus, if we can find a connected planar multigraph $\GGG $ for any Young diagram associated with an admissible partition of $(n, d)$, we may find a connected planar multigraph for any Young Diagram associated with an admissible partition of $(n, d+1)$, so our induction is complete, and Theorem \ref{thm:graphtheory} is proved. 
\end{proof}

The same ideas in fact show that if the assumption $n \leq d$ is dropped, one may still find that the branch data arises from the blowing up construction of a now not necessarily connected multigraph.  

\begin{theorem}
\label{thm:genrealizable}
Suppose $d \geq 2$ and branch data given by the partition  $k_1+\ldots + k_n=2d-2$, where $1 \leq k_i \leq d$ for $i=1, \ldots, n$, is given.  Then there exists a planar multigraph $\GGG$ such that blowing up the edges of $\GGG$ yields a critically fixed Thurston map, and hence a branched covering, realizing  this branch data.
\end{theorem}

\pf If $d=2$ we blow up a single edge.  We now induct on the integer $d$ and adapt the setup of the proof of Theorem \ref{theorem:ptomg}. At the inductive step, we have a Young diagram whose corresponding partition is given by a planar multigraph $\GGG$.  As before, the new Young diagram falls into one of three cases. 

If two new rows, each with one box, are added, we add to $\GGG$ a disjoint copy of a segment (two vertices joined by a single edge). 

If a box is added to two distinct existing rows, we join the corresponding vertices of $\GGG$ by an edge. 

If one new row with two boxes is added, we do the following. The new row will correspond to a new vertex $v$.  Add $v$ to the interior of a face $F$ of $\GGG$. There are at least two vertices $a, b \in \bdry F$ which are joined by an edge $e$ of $\GGG$. We delete $e$ and add two new edges, one joining $v$ and $a$, and another joining $v$ and $b$.  The valence at $a, b$ is preserved, and the valence of $v$ is equal to two as desired. 
\qed

\section{Tischler graphs} 
\label{secn:tischler}

Suppose $f$ is a critically fixed rational map. Suppose $c \in \C(f)$ has multiplicity $m$ and let $\Omega_c$ be its immediate basin.  Up to precomposition by multiplication by $m$th roots of unity, there is a unique holomorphic isomorphism $\Phi: (\D, 0) \to (\Omega_c, c)$ conjugating $w \mapsto w^{m+1}$ on $\D$ with $f$ on $\Omega_c$ which extends continuously to a semiconjugacy on the closures.  The image $\Phi(r\exp(2\pi i t)), r \in [0,1]$ is the {\em internal ray of angle} $t \in \R/\Z$ in $\Omega_c$; when $c=\infty$ these are called {\em external rays}. Note that the ray of angle $t$ is fixed if and only if $t = j/m \bmod 1$ for some $j\in\{0, 1, \ldots, m-1\}$. 

The {\em Tischler graph} is a planar embedded graph; as a subset of $\P^1$ it is the union of the fixed internal rays as in the previous paragraph.

\begin{lemma}
\label{lemma:tischlergraph}
Let $\TTT=(V, E)$ be the Tischler graph of a critically fixed rational map $f$ of degree $d$. Then $f(\TTT)=\TTT$, $f|_\TTT \simeq \id_\TTT$ relative to $C\sqcup R$, and the following hold: 
\begin{enumerate}
\item $\TTT$ is bipartite, with $V=C \sqcup R$ and edges joining elements of $C$ to elements of $R$
\item $|E| = 2d-2$
\item $|C|+|R| \leq d+1$
\item $2 \leq |C|$
\item $|C| \leq d$ 
\item for any 2-cycle in $\TTT$, each of the complementary disks meets a point in $C$. 
\end{enumerate}
\end{lemma}

\pf The invariance properties and the first five statements are clear. To prove the last, suppose distinct edges $e_\pm$ join a critical point $c$ of multiplicity $m$ to $r$ and $e_+ \union e_-$ bounds a Jordan domain $D$ whose interior does not meet  $C\union R$.  Then $f: D \to D$ maps by degree $1$.  But this is impossible, since the difference in angles between the rays comprising $e_+$ and $e_-$ must be at least $1/m$ which is larger than $1/(m+1)$. 
\qed

\subsection{From multigraphs to Tischler graphs}
\label{subsecn:mg_to_tg}
In this subsection, we show that if $F_\GGG$ is obtained by blowing up a connected planar multigraph $\GGG$, then the Tischler graph $\TTT_f$ of the rational map $f$ equivalent to $F_\GGG$ may be reconstructed from $\GGG$. 

Fix a topological realization of such a graph $\GGG$; we emphasize that multiple edges are drawn as disjoint.  We therefore allow {\em bigons} as faces of the cell structure of $S^2$ determined by $\GGG$.   For each face of $\GGG$, choose a point $p$ in its interior, and denote by $R_F$ the resulting collection of points (we will construct $F$ momentarily).  Put $C_F=V(\GGG)$.   Consider the identity map, as an augmented (``marked'') branched covering with marked set $C_F \union R_F$.   

Let $F$ be the result of blowing up this marked branched covering along the edges of $\GGG$, as in \cite{kmp-tan:blow}.   By {\em ibid.}, Proposition 2, the resulting combinatorial class is well-defined.  We claim that $F$ is combinatorially equivalent, as a marked branched map, to a rational map $f$, with $R_f$ corresponding to $R_F$ and $C_f$ to $C_F$.  In particular, $f$ has a repelling fixed point for each face of $\GGG$.  
Since we have assumed throughout that $\#C \geq 2$, the orbifold of $F$ is hyperbolic, so it is sufficient to rule out obstructions.  But this is clear, since no obstruction can meet an edge of $\GGG$, and each face contains only one element of $R_F$.  

In this paragraph, we construct a combinatorial candidate for the Tischler graph.  Consider the sphere $S^2$ with its cell structure induced from $\GGG$.  Let $D^2$ denote the closed unit disk, let $D$ be a face of $\GGG$, and let $\psi_D: D^2 \to D$ be an attaching map.  We may assume that $\psi_D(0)\in R_F$.  Consider the subdivision of $D$ obtained by joining the origin to the $0$-cells on $\bdry D^2$ via what we will suggestively call ``internal'' $1$-cells and mapping the result forward by $\psi_D$.  We obtain in this way a bipartite graph $\TTT_F$ whose edges are internal $1$-cells, and whose vertices are the disjoint union of $R_F \union C_F$.  

\begin{theorem}
\label{thm:mg_to_tg}
Suppose $\phi: (S^2, C_F \union R_F) \to (\rs, C_f \union R_f)$ is a combinatorial equivalence from $F$ to $f$,  Then $\phi(\TTT_F)$ is isotopic relative to $C_f \union R_f$ to the Tischler graph of $f$.  
\end{theorem}

\pf  The ideas are by now standard; we merely sketch the main ingredients.  By conjugating $F$ with $\phi$ we may assume $C=C_F=C_f, R=R_F=R_f$, $\phi \simeq \id$ relative to $C \union R$, and $\TTT_F \subset \rs$; we replace $F$ with this conjugate, so that $F: \rs \to \rs$.  
By construction, each $c \in C_f$ is incident to $m(c)$ edges in $\TTT_F$ and in $\TTT_f$, where $m(c)$ is the multiplicity of $c$.    It follows that $F$ is homotopic relative to $C \union R$ to a map $F'$ which agrees with $f$ near $C$; we replace $F$ with $F'$ and denote the result by $F$ again.  Furthermore, we may assume $\TTT_F=\TTT_f$ near $C$.  Since $F(\TTT_F)=\TTT_F$, there is a preferred subgraph of $f^{-1}(\TTT_F)$ isotopic to $\TTT_F$ via a family of homeomorphisms $h_t, t \in [0,1]$.  

Inductively, we may lift the isotopy $h_t, t \in [0,1]$ under iterates of $f$ to a real one-parameter family $h_t, t \in [0,\infty)$.  Since $f^{-1}$ is uniformly contracting away from $C$, the family $h_t$ converges as $t \to \infty$, yielding an isotopy from $\TTT_F$ to $\TTT_f$.  
\qed

\subsection{From Tischler graphs to multigraphs} 
\label{subsecn:tg_to_mg}
Suppose $f$ is a critically fixed rational map which arises from blowing up a planar multigraph $\GGG$, and let $\TTT$ be its Tischler graph, as constructed in the preceding subsection.  Thinking of edges of $\GGG$ as black, and edges of $\TTT$ as red, we observe the following.  The union $\GGG \union \TTT$ gives a cell structure to the sphere in which each cell is a triangle bounded by two red edges and one black edge,   with the red edges incident to a common vertex in $R_f$, and the black edges incident to two distinct vertices in $C_f$.   Moreover, each black edge $e$ lies in the boundary of exactly two such triangles, which therefore form a quadrilateral.  

\begin{figure}
\begin{center}
\includegraphics[width=3in]{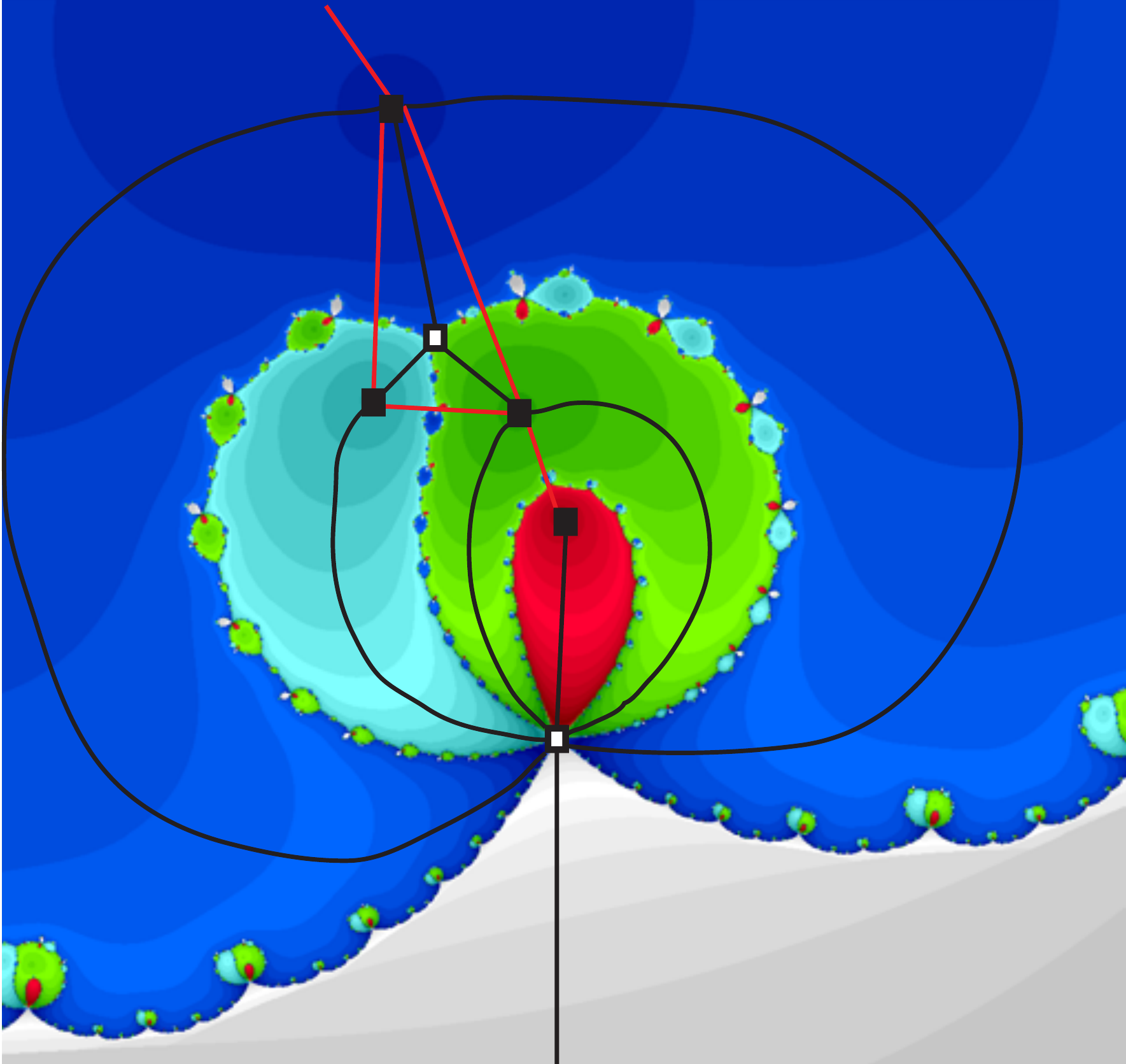}

{\sf Figure 1.  Detail of Example {\bf 33211a} of \S 11.  Infinity is a fixed critical point.  Shown in black is the Tischler graph $\TTT$ and in red the graph $\GGG$.}  
\end{center}
\end{figure}

This black edge $e$, which is a diagonal of this quadilateral, has two distinguished preimages: one each obtained by lifting the homotopy to either ``side'' formed by the union of two red edges.  By construction, each such  preimage $\tilde{e}$ of $e$ is homtopic to $e$ relative to $C(f)$.   

 This observation shows how $\GGG$ may be reconstructed from $\TTT$: the Tischler graph divides the sphere into quadrilaterals; the boundary of each  quadrilteral has precisely two vertices in $C_f$; the graph $\GGG$ is obtained from $\TTT$ by joining, for each quadrilateral,  this pair of vertices by a black edge.  To see this, simply note that each black edge has two preimages, each one obtained by lifting a homotopy joining this edge to one side of the quadrilateral in which it lies.  We conclude: 

\begin{theorem}
\label{thm:tg_to_mg}
Suppose $f$ is a critically fixed rational map with the property that its Tischler graph $\TTT_f$ is connected.  Then $f$ is equivalent to $f_\GGG$, where $\GGG$ is the planar multigraph obtained from $\TTT_f$.   
\end{theorem}

\subsection{Every Tischler polynomial is a blown-up star}
\label{subsecn:star}

As a special case, we have 

\begin{corollary}
\label{cor:stars}
Suppose $f$ is a a polynomial of degree $d$ each of whose critical points is also a fixed-point.  Then $f$ is equivalent to $F_\GGG$, where $\GGG$ is a ``star'' graph with $d-1$ distinct edges incident to a common vertex.  
\end{corollary}

\section{Subdivision rules}

Cannon, Floyd, and Parry \cite{cfp:fsr} have formulated the notion of a {\em finite subdivision rule}.  In our setting, these correspond to Thurston maps $f: (S^2,P) \to (S^2, P)$ which preserve some additional structure: there exist cell structures $X_0, X_1$ on $S^2$ such that $X_1$ refines $X_0$ and the restriction of $f$ to each cell of $X_1$ is a homeomorphism.  

If $f$ is a rational map equivalent to $F_\GGG$ for some graph $\GGG$, then its the Tischler graph $\TTT$ forms the 1-skeleton of a cell structure $X_0$ on the sphere.  This structure may be lifted under $f$ to obtain a structure cell structure  $X_1$ such that $f$ becomes the subdivision map of a finite subdivision rule; see Figure 5.  

\section{Examples} 
\label{secn:examples} 

Critically fixed polynomials are already classified by their Tischler graphs; equivalently, by the corresponding star graph.  We therefore focus exclusively on the non-polynomial cases.  
After classifying by degree, we classify first by the corresponding partition of $2d-2$; the exclusion of polynomials implies that we need consider only those partitions of $2d-2$ whose largest part is at most $d-2$.  
Next we classify by abstract graphs, and then by planar multigraph.   

In degrees 2 and 3, the only examples are conjugate to polynomials, as is easily verified.  The results given below for degree 4 and 5  were obtained by using the software package Maple to find all rational maps realizing the given partition, computing the number of conjugacy classes, and then comparing this count with the count of planar multigraphs.  Below, by a ``sticker'' we mean an edge of a graph with a vertex of valence 1. 
The assertions for degrees 4 and 5 below follow from solving the equations arising from the critical orbit portraits as in \cite{kmp:census}. 

Before turning to the data, we pause to make mention of a few special subfamilies.  

\subsection*{Newton's method} 
By considering the difference $f(z)-N_p(z)$, it is easy to see that any critically fixed rational map $f$ of degree $d$ with a repelling fixed-point at infinity and $d$ distinct fixed finite critical points $c_1, \ldots, c_d$ is equal to the Newton's method $N_p(z)$ applied to the polynomial $p(z)=(z-c_1)\cdots(z-c_d)$.   So if $\GGG$  is any connected multigraph that is a tree and has no multiple edges, then the rational map $f_\GGG$ is conjugate to a Newton's method.  Such examples are very special among Newton methods. 

\subsection*{Belyi maps} 
When $\Crit(f)=3$, any such map is a blown-up triangle.  Since $\#V(f)=3$, the map $f$ is a so-called {\em Belyi map}.  It seems likely that there should be a closed-form formula for $f$ in terms of the local degrees at points in $C(f)$.  
 
\newpage

\subsection*{Degree 4} 

\noindent{$\mathbf{[2,2,2]}$} There is a unique such example, obtained by blowing up the edges of a graph whose underlying set is a triangle; it is conjugate to $z \mapsto z^3(z-2)/(1-2z)$. McMullen \cite[Prop. 1.2]{ctm:iterative} uses this to give a generally convergent iterative algorithm for approximating the (distinct) roots of the general cubic polynomial.  Curiously, it also arises as one of the two maps comprising the correspondence on moduli space induced by the generic cubic rational function; see \cite{bekp} and \cite{lodge:thesis}.  Even more curiously, up to pre- and post-composition with M\"obius transformations, it appears to be the unique non-bicritical rational map whose fibers are conformally rigid; this observation is due to A. Saenz.  To see this, simply notice that as $w \to c \in C$, the fiber above $w$ approaches a configuration which has nearly a symmetry of order 3, so that the $j$-invariant of this fiber is a bounded holomorphic function of $w$.  
\gap

\noindent{$\mathbf{[2,2,1,1]}$} There is a unique such example, obtained by blowing up a planar segment of length 3; it is conjugate to a Newton method.  Up to conjugacy it is given by 
\[ z \mapsto -{\frac {{z}^{3} \left( 2\,z+3\,zc-3-4\,c \right) }{z+c}}\]
where $c=-3/8\pm \sqrt{5}/8$.

\subsection*{Degree 5}\ \\

\noindent{$\mathbf{[3,3,2]}$} There is a unique such example, obtained by blowing up a triangle with a single doubled edge. It is given by 
\begin{equation}
\tag{5.1}
 \hskip 0.5in z \mapsto -{\frac {{z}^{4} \left( 3\,z-5 \right) }{5\,z-3}}.
 \end{equation}
\gap

\noindent{$\mathbf{[3,3,1,1]}$} As an abstract graph, there is a unique such example: a segment of length three with its central edge doubled.  However, as a planar graph, there are two possibilities: (i) a planar segment with a doubled central edge, and (ii) a bigon with a sticker on both the outside and inside. 

They are given by
\[ z\mapsto -{\frac {{z}^{4} \left( 3\,z+4\,zc-4-5\,c \right) }{z+c}}\]
where in case (i) we have 
\begin{equation}
\tag{5.2} c=-1/2\pm \sqrt{5}/10
\end{equation}

and where in case (ii) we have 
\begin{equation}
\tag{5.3}
c=-3/10\pm \sqrt{21}/10.  
\end{equation}
Galois conjugate pairs turn out to yield M\"obius conjugate maps.  
\gap

\noindent{$\mathbf{[3,2,2,1]}$} There are two graphs realizing this portrait: 
(i) a segment of length 3 with an end segment doubled, yielding 
\begin{equation}
\tag{5.4}  z \mapsto \frac{z^3(16.39759477z^2-43.99398693z+33.32932462)}{6z-.267067538}, 
\end{equation} 
and (ii) a triangle with a sticker, yielding 
\begin{equation}
\tag{5.5}  z \mapsto \frac{z^3(1.02035014Iz^2-(3-2.55087534I)z+6.+1.70058356I)}{6z-3.+.17005836I}.
\end{equation}

\gap

\noindent{$\mathbf{[2,2,2,2]}$} There is a unique such example, obtained by blowing up the edges of a square.  Up to conjugacy it is given by 
\begin{equation}
\tag{5.6}
 z \mapsto -\frac{z^3(-\frac{3}{5}z^2+3z-3)}{-3z^2+6z-\frac{12}{5}}.
 \end{equation} 
\gap

\noindent{$\mathbf{[3,2,1,1,1]}$} There seems to be a unique such example: a letter ``Y'' with vertical segment of length 2, and no doubled edges.  It is conjugate to a Newton's method.  Up to conjugacy it is given by the formula 
\begin{equation}
\tag{5.7} 
z \mapsto \frac{z^3(0.00391776z^2-0.06348335z+.41144821)}{z-.64811739}.
\end{equation} 
  Finding this map algebraically is formidable: with one normalization, one is left with a system of equations which reduces to solving a sextic univariate polynomial; it can be shown that certain pairs of roots yield conjugate maps, while other pairs yield the same map. 
\gap

\noindent{$\mathbf{[2,2,2,1,1]}$} The graph is a planar segment of length 4, with no doubled edges. 
There is a unique such example, given by the formula 
\begin{equation}
\tag{5.8} 
{\small z \mapsto {\frac {{z}^{3} \left( -\frac{1}{20}\,{z}^{2}-\frac{1}{4}\,z+1 \right) }{{z}^{2}-\frac{1}{4}\,
z-\frac{1}{20}}}.}
\end{equation} 
Shown below are the corresponding Julia sets of the maps (5.1)-(5.8); the colors correspond to the basins of attraction. 
\begin{center}
\includegraphics[width=5in]{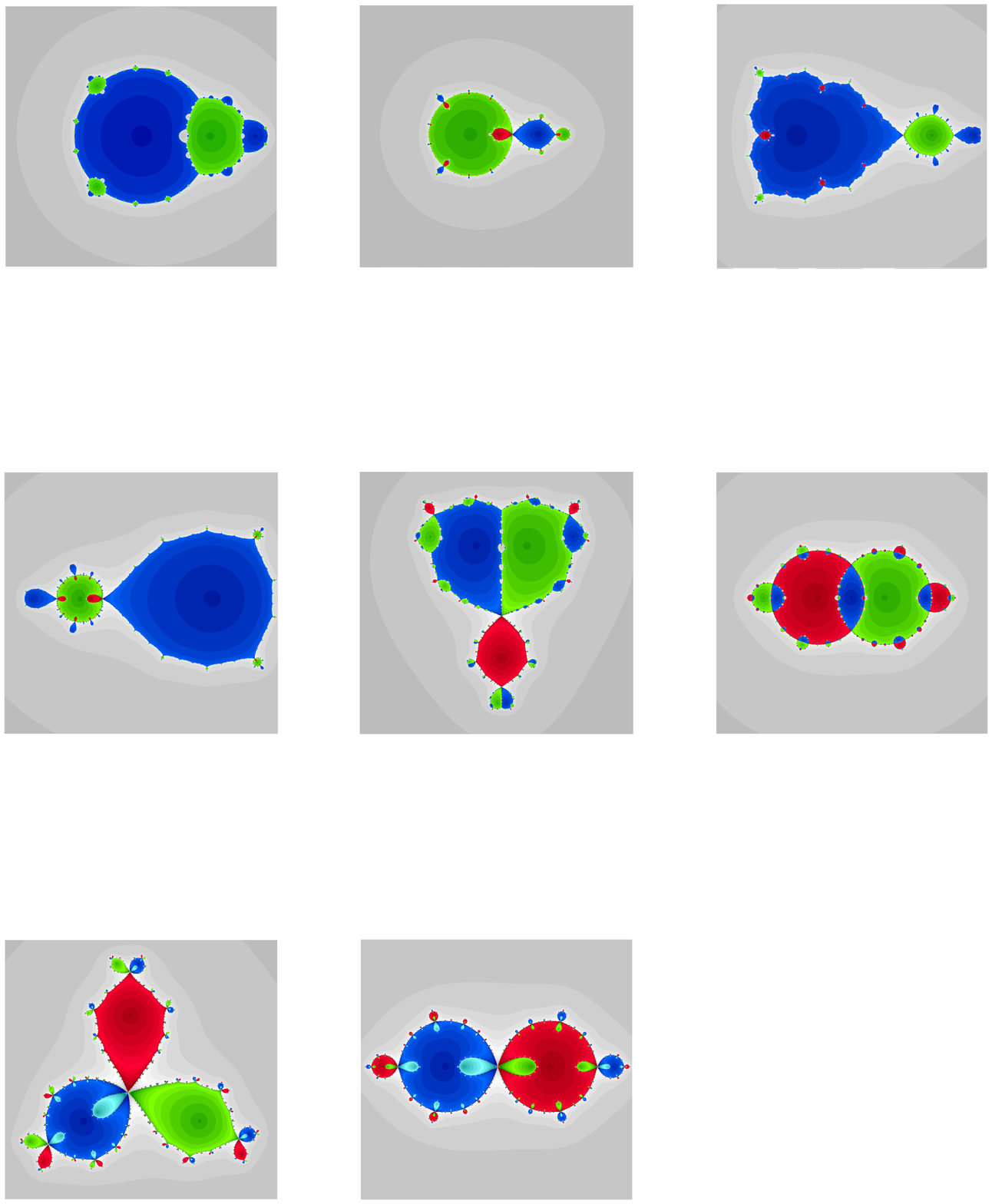}
{\sf Figure 2.  From upper-left to lower-right: Julia sets of the maps (5.1)-(5.8).  Pixels are colored according to their basin of attraction.}  
\end{center}

\newpage

\subsection*{Degree 6} 

Given the complexity encountered in the enumeration of quintics via naive algebraic methods, we take a different approach to the enumeration of sextics. 

Our first task is to sensibly enumerate the corresponding partitions.  Within maps of a fixed degree, one naive way to do this is to note that given any positive integer $N$, the set of partitions of $N$ is totally ordered by regarding the parts as digits in a fixed base $B$ larger than $N$ and computing the resulting integer.  We obtain using e.g. $B=10$: 

\begin{center} 
\noindent{$\mathbf{421111}$}

\noindent{$\mathbf{331111}$}

\noindent{$\mathbf{322111}$}

\noindent{$\mathbf{222211}$}

\noindent{$\mathbf{43111}$}

\noindent{$\mathbf{42211}$}

\noindent{$\mathbf{33211}$}

\noindent{$\mathbf{32221}$}

\noindent{$\mathbf{22222}$}

\noindent{$\mathbf{4411}$}

\noindent{$\mathbf{4321}$}

\noindent{$\mathbf{4222}$}

\noindent{$\mathbf{3331}$}

\noindent{$\mathbf{3322}$}

\noindent{$\mathbf{442}$}

\noindent{$\mathbf{433}$}
\end{center}
Note that the first four will correspond to Newton methods, and the last two to Belyi maps.  

Instead of systematically tabulating data, we discuss an example.  
\gap

The partition $\mathbf{33211}$ corresponds to three abstract isomorphism and four planar isomorphism classes.  The four planar multigraphs are shown below, along with a corresponding wreath recursion.  In each, the $i$th generator surrounds the $i$th vertex, but further dependence of the wreath recursion on the precise choice of the generators and connecting paths has been suppressed.   For the first graph, we give also the numerical approximation given by Bartholdi's program, the approximate locations of the attractors, and an image of the Julia set (Figure 3); detail is shown in Figure 1.  The first two tables give the coefficients of the numerator and of the denominator.  

\newpage

\subsection*{33211a}
\[\begin{tikzpicture}[scale=1]
\node(p1) at (-1,.5) {$\bullet$};
\node(p2) at (1,.5) {$\bullet$};
\node(p3) at (0,-1.2) {$\bullet$};
\node(p4) at (-2,1) {$\bullet$};
\node(p5) at (2,0) {$\bullet$};

\draw (-1,.5) node [above]{4} -- (1,.5) node [above]{1} -- (0,-1.2) node [below] {3} -- cycle;
\draw (1,.5) -- (2,0) node [right] {2};
\draw (-1,.5) -- (-2,1) node [left]{5};
\node[right] at (4,0) {
      \begin{minipage}{.4\textwidth}
\begin{align*}
\Phi(a) & = \langle abc, 1, b^{-1}, c^{-1}, 1, 1 \rangle (1432)\\
\Phi(b) & = \langle 1, b, 1, 1, 1, 1 \rangle (12)\\
\Phi(c) & = \langle 1, 1, c, 1, 1, 1 \rangle (135)\\
\Phi(d) & = \langle 1, 1, 1, d, 1, 1 \rangle (1546)\\
\Phi(e) & = \langle 1, 1, 1, 1, 1, e \rangle (16)\\\\
\end{align*}
      \end{minipage}
   };
\end{tikzpicture}\]

\[ \begin{array}{cc} 
\mbox{\bf numerator} & \mbox{\bf denominator} \\ \hline 
%
\begin{array}{c|r|r}
\mbox{\bf degree} & \mathbf{\Real} & \mathbf{\Imag} \\ \hline 
0 & -.2823382388 & .0873666659 \\ \hline 
1 &  -.6106212003& -.2210949962\\ \hline
2 & -1.5258285009& 1.6871931450\\ \hline 
3 & 3.3728946411& -2.2256057710\\ \hline 
4 & -.5227487239 & 3.1529307220\\ \hline
5 & -8.5737682757& .3796365169 \\  \hline
6 & -.3915110282& -1.2571815540\\\hline
\end{array}
& 
\begin{array}{c|r|r}
\mbox{\bf degree} & \mathbf{\Real} & \mathbf{\Imag} \\ \hline 
0 & -1.0000000000 & 0 \\ \hline 
1 & -3.6399623589 & .5053341577 \\   \hline 
2 & 1.8652782469 & -2.2777114530\\   \hline 
3 &  -2.2515718241 & 7.7088079910 \\   \hline 
4 & -9.2611252563 & .6459068671 \\   \hline 
5 & 0 & 0 \\    \hline 
6 & 0 & 0 \\   \hline 
\end{array}
\end{array}
\]
\[ 
\begin{array}{c|r|r}  
\mbox{\bf vertex} & \mathbf{\Real} & \mathbf{\Imag}  \\  \hline 
 p_1 & 0.26444 & -0.046607\\  \hline 
p_2 & 0.52144& -0.66036 \\   \hline 
p_3 & -0.59233& 0.023373\\   \hline 
p_4 & -0.47002& 1.7073\\   \hline 
p_5 & \infty  & \infty \\  \hline 
\end{array}
\] 

\begin{center}
\includegraphics[width=2.5in]{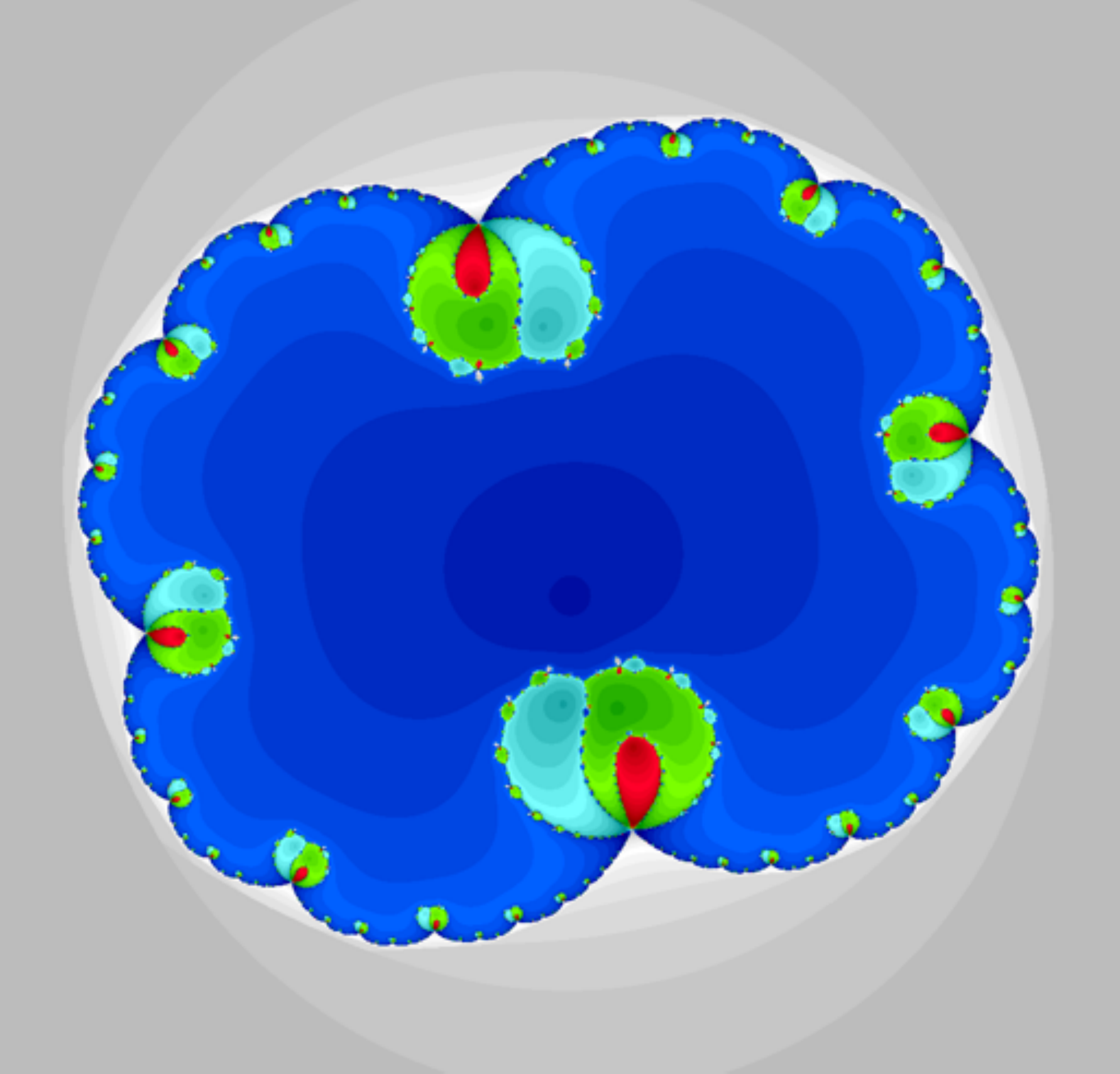}

{\sf Figure 3.  The closures of the five immediate basins meet at a repelling fixed-point, located near the bottom of the image, and three (light blue, green, dark blue) also meet at another repelling fixed-point, a bit higher up in the image.  Detail is given in Figure 1.}  
\end{center}

\newpage
\subsection*{33211b}
\[\begin{tikzpicture}[scale=1]
\node(p1) at (-1,-.5) {$\bullet$};
\node(p2) at (1,-.5) {$\bullet$};
\node(p3) at (0,1.2) {$\bullet$};
\node(p4) at (-2,-1.5) {$\bullet$};
\node(p5) at (0,.1) {$\bullet$};

\draw (-1,-.5) node [left]{4} -- (1,-.5) node [right]{2} -- (0,1.2) node [above] {1} -- cycle;
\draw (0,.1) node [right]{3} -- (0,1.2);
\draw (-1,-.5) -- (-2,-1.5) node [left]{5};
\node[right] at (4,0) {
      \begin{minipage}{.4\textwidth}
\begin{align*}
\Phi(a) & = \langle abc, 1, b^{-1}, c^{-1}, 1, 1 \rangle (1432)\\
\Phi(b) & = \langle 1, b, 1, 1, 1, 1 \rangle (125)\\
\Phi(c) & = \langle 1, 1, c, 1, 1, 1 \rangle (53)\\
\Phi(d) & = \langle 1, 1, 1, d, 1, 1 \rangle (1546)\\
\Phi(e) & = \langle 1, 1, 1, 1, 1, e \rangle (16)\\\\
\end{align*}
      \end{minipage}
   };
\end{tikzpicture}\]

\subsection*{33211c}
\[\begin{tikzpicture}
\node(p1) at (1.5,1.5) {$\bullet$};
\node(p2) at (0,3) {$\bullet$};
\node(p3) at (-1.5, 1.5) {$\bullet$};
\node(p4) at (-3,0) {$\bullet$};
\node(p5) at (-4.5,-1.5) {$\bullet$};
\draw (1.5,1.5) node [right]{2} -- (0,3);
\draw (-1.5,1.5) node [above]{3}-- (-3,0)node [above]{4} (-4.5,-1.5) node[above]{5};
\draw (0,3) node [above]{1} to [bend left=14] (-1.5,1.5) to [bend left=14] (0,3);
\draw (-1.5,1.5) -- (-4.5,-1.5);
   \node[right] at (4,0) {
      \begin{minipage}{.4\textwidth}
\begin{align*}
\Phi(a) & = \langle ab, 1, 1, b^{-1}, 1, 1 \rangle (1423)\\
\Phi(b) & = \langle 1, b, 1, 1, 1, 1\rangle (132)\\
\Phi(c) & = \langle 1, 1, 1, c, 1, 1 \rangle (145)\\
\Phi(d) & = \langle 1, 1, 1, 1, d, 1 \rangle (156)\\
\Phi(e) & = \langle 1, 1, 1, 1, 1, e \rangle (16)\\\\
\end{align*}
      \end{minipage}
   };
\end{tikzpicture}\]

\subsection*{33211d}

\[\begin{tikzpicture}
\node(p1) at (0,-2) {$\bullet$};
\node(p2) at (1.5,-1.5) {$\bullet$};
\node(p3) at (0, 0) {$\bullet$};
\node(p4) at (-1.5,-1.5) {$\bullet$};
\node(p5) at (-3,-3) {$\bullet$};
\draw (0,-2) node [left]{3} -- (0,0);
\draw (0,0) node [above]{1}-- (-1.5,-1.5)node [right]{4} (-3,-3) node[below]{5};
\draw (1.5,-1.5) node [above]{2} -- (0,0);
\draw (-1.5,-1.5) to [bend left=14] (-3,-3) to [bend left=14] (-1.5,-1.5);
   \node[right] at (4,0) {
      \begin{minipage}{.4\textwidth}
\begin{align*}
\Phi(a) & = \langle abc, 1, b^{-1}, c^{-1}, 1, 1 \rangle (1432)\\
\Phi(b) & = \langle 1, b, 1, 1, 1, 1\rangle (12)\\
\Phi(c) & = \langle 1, 1, c, 1, 1, 1 \rangle (13)\\
\Phi(d) & = \langle 1, 1, 1, d, 1, 1 \rangle (1456)\\
\Phi(e) & = \langle 1, 1, 1, 1, e, 1 \rangle (165)\\\\
\end{align*}
      \end{minipage}
   };
\end{tikzpicture}\]

Reflecting on the fact that wreath recursions depend on choices, it is perhaps remarkable that the preceding examples have such a simple form.  Though perhaps hard to define ``simple'' precisely, indeed, this is not an accident.  In \cite{np:monodromy},  an algorithm is given which takes, as input, a connected planar multigraph $\GGG$, together with certain additional data, and returns as output a wreath recursion for the corresponding Thurston map $f_\GGG$.  This algorithm was effectively implementable by hand, resulting in the efficient calculation of the wreath recursions for all quartic, quintic, and sextic examples.  

In the following section, we mention below a few key features of this algorithm.  Suppose $\GGG$ is a planar multigraph with $n=\#V(G)$ vertices, and $f=f_\GGG$ is the corresponding Thurston map. 

\section{Computation of wreath recursions} 
\label{secn:wr_computation}
In outline, here is the algorithm which was employed for the computation of a wreath recursion associated with the maps $f_\GGG$.  

\begin{enumerate}
\item Begin with a topological realization of $\GGG$ (that is, drawing multiple edges as parallel), and chooses a spanning tree $T$ for the planar dual $\GGG^*$.  Choose a basepoint vertex $b$ for $T$. 

\item Choose a set of arcs $r_i, i=1, \ldots, n$ from $b$ to the elements of $V$ as in \S 2.8 so that in addition the following hold: (i) no arc $r_i$ crosses an edge of $\GGG$ more than once; (ii) no arc $r_i$ crosses an edge incident to its terminal vertex.  

That this is possible follows easily by the greedy algorithm.  Initialization: let $F$ be the face of $\GGG$ containing $b$; join $b$ to each vertex on $\bdry F$.  Induction: suppose there exists a face $F$ not all of whose vertices have been visited.  There exists a closest such face to $b$, as measured by the edge-path length in $T$.  Choose one such face $F$, run from $b$ to $F$ along edges of $T$, and then run to the unique vertex on $\bdry F$ which is not yet visited and which is first in the counterclockwise ordering of these vertices, starting from the vertex immediately to the right of the traversed edge-path in $T$ as it enters $F$.   

Label these arcs $r_i$ in the usual counterclockwise cyclic fashion, and label their terminal points $v_i$.  

\item Next, label the remaining preimages of the basepoint $b_2, \ldots, b_d$ as follows.

The map $f$ is obtained by ``blowing up'' the edges $e_j$ to disks $D_j$, where $f: D_j \to S^2-e_j$ is a homeomorphism.  Thus each $D_j$ contains a unique element $b_j \in f^{-1}(b)$.  We may assume, for convenience of representation, that $b_j$ is drawn as a midpoint of $e_j$.  We may identify $f^{-1}(b) - \{b\}$ with the set of edges of $\GGG$.  Set $b_1=b$.  
For each $v_i, i=1, \ldots, n$, do the following. Examine the neighbors of $v_i$.  Find the neighbor with smallest index, $j$, for which there exists a nonempty set $E$ of unlabelled edges $e$ joining $v_i$ and $v_j$; if no such neighbor exists, proceed to $v_{i+1}$.  Otherwise, the set $E \union \{r_i\}$ has a total order, given by the counterclockwise ordering in which they are incident to $v_i$, with $r_i$ as the least element.  Label the elements of $E$ consecutively, with increasing indices on the labels  $b_k$'s.

The proof of the following theorem uses the way in which the arcs $r_i$ were chosen, but not the way in which the labels of edges (preimages of $b$) were assigned.  To describe the result, we imagine that the subset of the arcs $r_i$ which join the basepoint $b$ to vertices $w$ on the boundary of the face containing $b$ are ``labelled'' by the symbol $1$, and for the purposes of describing the monodromy, we imagine that these arcs are new ``edges'' incident to such vertices $w$.  See Figures 4 and 8 for a clarifying example.  

\begin{theorem}
\label{thm:simple_monodromy}
The monodromy induced by $g_i$ is the pure cycle consisting of the indices of the edge labels incident to $v_i$, with the counterclockwise cyclic order.
\end{theorem}

\item Next, choose connecting arcs as follows.  Set $\lambda_1$ to be the constant path at $b_1=b$.  Given $j>1$, the edge $e_j$ is incident to two vertices $v_k, v_\ell$.  Let $i=\max\{k, \ell\}$, and choose $\lambda_j$ so that it follows $r_i$ to $v_i$, loops around $v_i$ counterclockwise in the direction of $g_i$ until it reaches $e_j$, and then follows $e_j$ to $b_j$.  

\end{enumerate}

\[
                \begin{tikzpicture}[scale = 1]
                \node [draw, circle, fill=gray!8]  (v1) at (0,4) {$v_1$};
                \node [draw, circle, fill=gray!8] (v2) at (0,0){$v_2$};
                \node [draw, circle, fill=gray!8] (v3) at (2,-3.5){$v_3$};
                \node [draw, circle, fill=gray!8] (v4) at (-2,-3.5) {$v_4$};
                \node (b1) at (4, -3) {$\blacksquare$};
                
                \draw (b1) node [right]{~$b$};
                \draw (v1) -- node[right]{$e_2$}  (v2);
                \draw (v2) -- node[left]{$e_5$} (v4);
                \draw (v4) -- node [below]{$e_6$}(v3);
                \draw (v2) to [bend right = 10] node[left]{$e_3$} (v3);
                \draw (v2) to [bend left = 10] node [right]{$e_4$}(v3);
                \draw [blue, ->] (b1) to [bend right = 15] node[right]{$g_1$}(0.495, 3.505) arc (315: 655: .7);
                \draw [blue, ->] (b1) to [bend right = 30] node [below]{$g_2$} (0.7, 0) arc(0: 345:.7cm);
                \draw [blue, ->] (b1) to [bend left = 5] node [above]{$g_3$}(2.7, -3.5) arc (0: 345: .7cm);
                \draw [blue, ->] (b1) to [bend left = 45] node [below]{$g_4$} (-1.505, -3.995) arc (315: 655: .7cm);
                \end{tikzpicture}
                \]
\begin{center}{\sf Figure 4.  Choosing the arcs $r_1, \ldots, r_4$ and the preimage (edge) labels $e_2, \ldots, e_6$.} 
\end{center}

\subsection*{A worked example}  Consider the multigraph, connecting arcs, and edge-labels as chosen in Figure 4.  
Figure 5 shows the preimages of the edges and basepoints under the corresponding branched covering, labelled by their images, and with the generators (but not their preimages) superimposed.  \\

                          \[  
                          \begin{tikzpicture}  [scale=1.3]
                            \node[draw, circle, fill=gray!8] (p1) at (0,4) {$v_1$};
                            \node [draw, circle, fill=gray!8](p2) at (0,0){$v_2$};
                            \node [draw, circle, fill=gray!8] (p3) at (2,-3.5){$v_3$};
                            \node [draw, circle, fill=gray!8](p4) at (-2,-3.5) {$v_4$};
                            
                            \node(b1) at (4,-3) {$\blacksquare$};
                            \node(b2) at (0,2.7) {$\blacksquare$};
                            \node(b5) at (-1.2,-2) {$\blacksquare$};
                            \node(b6) at (-.5,-3.7) {$\blacksquare$};
                            \node(b3) at (.6, -2.2) {$\blacksquare$};
                            \node(b4) at (1.7,-1.85) {$\blacksquare$};
                            
                            \draw (4,-3) node [right]{$~b_1$};
                            \draw (0,2.7) node [right]{$~b_2$};
                            \draw (-1.2,-2) node [right]{~$b_5$};
                            \draw (-.5,-3.7) node [right]{$~b_6$};
                            \draw (.6,-2.2) node [right]{$~b_3$};
                            \draw (1.7,-1.85) node [right]{$~b_4$};

                            \node (p5) at (-.3, 1.8){$\bullet$};
                            \node (p6) at (.3, 1.8){$\bullet$};
                            \draw (p2) to [bend left = 3] node[left]{$e_4$}  (-.3, 1.8) node [left]{$v_3$};
                            \draw (p2) to [bend right = 3] node[right]{$e_3$}  (-.3, 1.8);
                            \draw (-.3,1.8) -- node[above]{$e_6$}  (.3, 1.8) node[right]{$v_4$} -- node[right]{$e_5$} (p2);
                            \draw (p1) to [bend left=40] node [right]{$e_2$} (p2);
                            \draw (p2) to [bend left = 40] node [left]{$e_2$} (p1);
                            
                            \draw (p2) to [bend left=12] node[right]{$e_5$}(p4);
                            \draw (p4) to [bend left=50] node[left]{$e_5$} (p2);
                            \node (p7) at (-1.7, -1.6){$\bullet$};
                            \node (p8) at (-.7, -1.3){$\bullet$};
                            \draw (p2) to [bend right = 15] node [above]{$e_3~$} (-1.7, - 1.6);
                            \draw (p2) to [bend right= 11] node[below]{$~e_4$} (-1.7, -1.6) node [right]{$v_3$};
                            \draw (-1.7, -1.6) to [bend right= 13]  node[right]{$e_6$}(p4);
                            \draw (p2) --node[left]{$e_2$} (-.7, -1.3) node [below]{$v_1$};
                            
                            \draw (p4) to [bend left=10] node [above]{$e_6$}(p3);
                            \draw (p3) to [bend left=50] node [below] {$e_6$}(p4);
                            \node (p9) at (0, -4.1){$\bullet$};
                            \node (p10) at (0, -3.5){$\bullet$};
                            \draw (p4) to [bend right=12] node [above]{$e_5$} (0,-4.1) node [below]{$v_2$};
                            \draw (0,-4.1) to [bend right=11] node[above]{$e_4$} (p3);
                            \draw (0, -4.1) to [bend right = 16] node [below] {$e_3$} (p3);
                            \draw (0,-4.1) -- node[right]{$e_2$} (0,-3.5) node [right]{$v_1$};
                            
                            \draw (p3) to [bend left=45] node [left]{$e_3$} (p2);
                            \draw (p3) to [bend right = 11] node [left]{$e_3$}(p2);
                            \draw (p3) to [bend right=17] node [right]{$e_4$} (p2);
                            \node (p11) at (.8, -1.75){$\bullet$};
                            \node (p12) at (.3, -1.5){$\bullet$};
                            \draw (p2) to [bend right=3] node[right]{$e_5$} (.8, - 1.75) node[right]{$v_4$};
                            \draw (.8, -1.75) to [bend right = 3] node [right]{$e_6$} (p3);
                            \draw (p2) -- node [left]{$e_2$} (.3, -1.5) node [below]{$v_1$};
                            
                            \draw (p3) to [bend right=70] node[right]{$e_4$} (p2);
                            \node(p13)at (1.1, -1){$\bullet$};
                            \node(p14)at (1.9, -1.3) {$\bullet$};
                            \draw (p2) -- node[right]{$e_2$} (1.1,-1) node [right]{$v_1$};
                            \draw (p2) to [bend left=20] node [right]{$e_5$} (1.9, -1.3) node [right]{$v_4$};
                            \draw (1.9, -1.3) to [bend left=20] node [left]{$e_6$} (p3);
                            
                            \draw [blue, ->] (b1) to [bend right = 35] node[right]{$g_1$}(.433, 3.75) arc (330: 680: .5);;
                            \draw [blue, ->] (b1) to [bend right = 40] node [below]{$g_2$} (.4698, .171)arc(20: 360:.5cm);
                            \draw [blue, ->] (b1) to [bend left = 5] node [above]{$g_3$}(2.48296, -3.37059) arc (15: 360: .5cm);
                            \draw [blue, ->] (b1) to [bend left = 60] node [below]{$g_4$} (-1.829, -3.9698) arc (290: 630: .5cm);
                            \end{tikzpicture}
                          \]
\begin{center}{\sf Figure 5.  Preimages of edges and of the basepoint $b$.}\end{center}  

Consider the monodromy induced by $g_4$.  The local pictures, gleaned from the previous figure, are as follows.  
Near the branch point $v_4$ circled by $g_4$, we see a cycle:

        \[
        \begin{tikzpicture}[scale = 1]
        
        \node [draw, circle, fill = gray!8] (v4) at (0,0) {$v_4$};

        
        \node (b1) at (-2,-2) {$\blacksquare$};
        \node (b5) at (-.4, 2) {$\blacksquare$};
        \node (b6) at (2.4, 0) {$\blacksquare$};
        \node at (-2,-2) [right] {$~b_1$};
        \node at (-.4,2) [right] {$~b_5$};
        \node at (2.4,0) [right] {$~b_6$};
        
        \node (e51) at (0,4.6) {};
        \node (v3) at (-1.4, 2.8){$\bullet$};
        \node (e52) at (-3,4){};
        \node (e3) at (-2, 4.6){};
        \node (e4) at (-1.6, 4.4){};

        
        \draw (v4) to [bend left = 5] node [left] {$e_6$} (-1.4, 2.8) node [left] {$v_3$};
        \draw (v4) to [bend left = 20]  node [left] {$e_5$} (e52);
        \draw (v4) to [bend right = 10] node [right] {$e_5$}(e51);
        \draw (-1.4, 2.8) to [bend left = 6] node [left] {$e_3$} (e3);
        \draw (-1.4, 2.8) to [bend left = 1] node [right] {$e_4$} (e4);
        
        \node (e61) at (4,1) {};
        \node (e62) at (3, -3) {};
        \node (v2) at (4.4,-1.6) {$\bullet$};
        \node (v1) at (4.8,0) {$\bullet$};
        \node (e50) at (6, -2) {};

        
        \draw (v4) to [bend left = 15] node [above] {$e_6$} (e61);
        \draw (v4) to [bend right = 5]  node [below] {$e_5$} (4.4, -1.6) node [below] {$v_2$};
        \draw (4.4, -1.6) to [bend right = 1] (e50);
        \draw (4.4, -1.6) -- node [right]{$e_2$} (4.8,0) node [right] {$v_1$};
        \draw (v4) to [bend right = 20] node [below] {$e_6$} (e62);

        
        \draw [blue, ->] (-1.6,-1.8) -- node [right] {$~\tilde{g}_4[b_1]$} (-.3, -.5196) arc (240 : 350 : .6 cm) -- (2.1, -.1);
        \draw [red, ->] (2.1, .1) -- node [above, near start] {$~\tilde{g}_4[b_6]$} (.590885, .104189) arc (10 : 90 : .6 cm) -- (-.25, 1.8);
        \draw [green, ->] (-.47, 1.8) -- (-.205212, .563816) arc (110 : 220 : .6cm) -- node [left] {~$\tilde{g}_4[b_5]~$}(-1.85, -1.7);
        
        \end{tikzpicture}
        \]
\begin{center}{\sf Figure 6.  A cycle in the monodromy induced by $g_4$.}\end{center}
\gap
            
For elements $b_j$ in $f^{-1}(b)$ which correspond to edges $e_j$ not adjacent to $v_4$, we see that the interior of $D_j$ contains a preimage copy of $v_4$, and thus the lift $f^{-1}(g_4)[b_j]$ of $g_4$ based at $b_j$  circles this copy of $v_i$ and returns to $b_j$ within $D_j$ so that $j^{g_4}=j$:

\resizebox{!}{152pt}{
            \begin{tikzpicture}[scale = .9]
            
            \node [draw, circle, fill = gray!8] (v2) at (-2,-2) {$v_2$};
            

            \node (v11) at (-4,2) {};
            \node (v12) at (0,2) {};
            
            \node (v3) at (-2.9,1) {$\bullet$};
            \node (v4) at (-1.1,1) {$\bullet$};

            
            \node (b) at (-1.9, 2.5) {$\blacksquare$};
            \node at (-1.9,2.5) [right] {$~b_2$};

            
            \draw (v2) to [bend left = 20] node [left] {$e_2$} (v11);
            \draw (v2) to [bend right = 20] node [right] {$e_2$} (v12);
            \draw (v2) to [bend left = 5] node [left] {$e_4$} (-2.9,1);
            \draw (v2) to [bend right = 5] node [right] {$e_3$} (-2.9,1);
            \draw (v2) -- (-1.1,1) node [right] {$v_4$};
            \draw (-2.9,1) node [left] {$v_3$} -- node [above] {$e_6$} (-1.1,1);

            
            \draw [blue, ->] (-1.85, 2.2) -- (-1.4, 1.4) arc (135 : 480: .5cm) -- node [right] {$\tilde{g}_4[b_2]$} (-1.75, 2.3);

            
            \node [draw, circle, fill = gray!8] (vvv2) at (2.5,2.5) {$v_2$};
            \node [draw, circle, fill = gray!8] (vvv3) at (5.5,-2.5) {$v_3$};
            
            
            \node (vvv1) at (2.6, .5) {$\bullet$};
            \node (vvv4) at (4.5, .2) {$\bullet$};
            
            
            \draw (vvv2) --  node [right] {$e_2$} (2.6,.5) node [right] {$v_1$};
            \draw (vvv2) to [bend left = 6] node [left]{$e_5$} (4.5,.2) node [right] {$v_4$} to [bend left = 6] node [left] {$e_6$} (vvv3);

            \node (bbb3) at (3.5, -1) {$\blacksquare$};
            \node at (3.5,-1)  [right] {$~b_3$};
            
            
            \draw (vvv2) to [bend left = 40] node [right] {$e_3$} (vvv3);
            \draw (vvv2) to [bend right = 65]  node [left] {$e_3$} (vvv3);
            
            
            \draw [blue, ->] (3.75, -.8) -- (4.25, -.23302) arc (240 : 585: .5 cm) -- node [near start, left] {$\tilde{g}_4[b_3]$} (3.65, -.7);
            
            \node [draw, circle, fill = gray!8] (vv2) at (6.5,2.5) {$v_2$};
            \node [draw, circle, fill = gray!8] (vv3) at (9.5,-2.5) {$v_3$};
            
            
            \node (vv1) at (8.3, 1) {$\bullet$};
            \node (vv4) at (10.5, 1.3) {$\bullet$};
            \node (ex) at (11.5, 1.5) {};
            
            
            \draw (vv2) --  node [right] {$e_2$} (8.3, 1) node [right] {$v_1$};
            \draw (vv2) to [bend left = 35] node [below]{$e_5$} (vv4) node [right] {$v_4$} to [bend left = 35] node [left] {$e_6$} (vv3);

            \node (bb3) at (9.3, -.2) {$\blacksquare$};
            \node at (9.3, -.2) [right] {$~b_4$};

            
            \draw (vv2) to [bend left = 50] (ex) node [right] {$e_4$} to [bend left = 50] (vv3);
            \draw (vv2) to [bend left = 10]  node [left] {$e_4$} (vv3);
            
            
            \draw [blue, ->] (9.5, -.1) -- (10.25, .867) arc (240 : 585 : .5cm) -- node [left] {$\tilde{g}_4[b_4]$} (9.4, 0);
            
            \end{tikzpicture}}
\begin{center}{\sf Figure 7.  Fixed-points in the monodromy induced by $g_4$.}\end{center} 
         
\newpage   
Thus the monodromy induced by $g_4$ is given by the pure cycle $(165)$.   The calculation of the remainder of the monodromy is similar, yielding:

            \[\begin{tikzpicture}[scale = .7]
            \node [draw, circle, fill=gray!8]  (v1) at (0,4) {$v_1$};
            \node [draw, circle, fill=gray!8] (v2) at (0,0){$v_2$};
            \node [draw, circle, fill=gray!8] (v3) at (2,-3.5){$v_3$};
            \node [draw, circle, fill=gray!8] (v4) at (-2,-3.5) {$v_4$};
            \node (b1) at (4, -3) {$\blacksquare$};
            
            \draw (0,4) circle  (.7cm);
            \draw (0,0) circle  (.7cm);
            \draw (2,-3.5) circle (.7cm);
            \draw (-2, -3.5) circle (.7cm);
            
            \node at (0,4.65) [above] {$\sigma(g_1) = (12)$};
            \node at (-.6,.5) [left] {$\sigma(g_2) = (12534)$};
            \node at (3, -4.3) [right] {$\sigma(g_3) =  (1436)$};
            \node at (-2.7, -3.2) [left] {$\sigma(g_4) =  (165)$};

            \draw (b1) node [right]{~$b$};
                \draw (v1) -- node[right]{$e_2$}  (v2);
                \draw (v2) -- node[left]{$e_5$} (v4);
                \draw (v4) -- node [below]{$e_6$}(v3);
                \draw (v2) to [bend right = 10] node[left]{$e_3$} (v3);
                \draw (v2) to [bend left = 10] node [right]{$e_4$}(v3);
                \draw [blue] (b1) to [bend right = 15] node[right]{$e_1$}(v1);
                \draw [blue] (b1) to [bend right = 30] node [below]{$e_1$} (v2);
                \draw [blue] (b1) to [bend left = 5] node [above]{$e_1$}(v3);
                \draw [blue] (b1) to [bend left = 45] node [below]{$e_1$} (v4);
            \end{tikzpicture}\]
\begin{center}{\sf Figure 8.  The monodromy of $f$.}\end{center}

\newpage

Figure 9 shows the result of our algorithm for constructing connecting paths.  
\begin{center}
\resizebox{!}{360pt}{
        \begin{tikzpicture}  [scale=1.5]
        \node[draw, circle, fill=gray!8] (p1) at (0,4) {$v_1$};
        \node [draw, circle, fill=gray!8](p2) at (0,0){$v_2$};
        \node [draw, circle, fill=gray!8] (p3) at (2,-3.5){$v_3$};
        \node [draw, circle, fill=gray!8](p4) at (-2,-3.5) {$v_4$};
        
        \node(b1) at (4,-3) {$\blacksquare$};
        \node(b2) at (0,2.7) {$\blacksquare$};
        \node(b5) at (-1.2,-2) {$\blacksquare$};
        \node(b6) at (-.5,-3.7) {$\blacksquare$};
        \node(b3) at (.6, -2.2) {$\blacksquare$};
        \node(b4) at (1.7,-1.85) {$\blacksquare$};
        
        \draw (4,-3) node [right]{$~b_1$};
        \draw (0,2.7) node [right]{$~b_2$};
        \draw (-1.2,-2) node [right]{~$b_5$};
        \draw (-.5,-3.7) node [right]{$~b_6$};
        \draw (.6,-2.2) node [right]{$~b_3$};
        \draw (1.7,-1.85) node [right]{$~b_4$};

        \node (p5) at (-.3, 1.8){$\bullet$};
        \node (p6) at (.3, 1.8){$\bullet$};
        \draw (p2) to [bend left = 3] node[left]{$e_4$}  (-.3, 1.8) node [left]{$v_3$};
        \draw (p2) to [bend right = 3] node[right]{$e_3$}  (-.3, 1.8);
        \draw (-.3,1.8) -- node[above]{$e_6$}  (.3, 1.8) node[right]{$v_4$} -- node[right]{$e_5$} (p2);
        \draw (p1) to [bend left=40] node [right]{$e_2$} (p2);
        \draw (p2) to [bend left = 40] node [left]{$e_2$} (p1);
        
        \draw (p2) to [bend left=12] node[right]{$e_5$}(p4);
        \draw (p4) to [bend left=50] node[left]{$e_5$} (p2);
        \node (p7) at (-1.7, -1.6){$\bullet$};
        \node (p8) at (-.7, -1.3){$\bullet$};
        \draw (p2) to [bend right = 15] node [above]{$e_3~$} (-1.7, - 1.6);
        \draw (p2) to [bend right= 11] node[below]{$~e_4$} (-1.7, -1.6) node [right]{$v_3$};
        \draw (-1.7, -1.6) to [bend right= 13]  node[right]{$e_6$}(p4);
        \draw (p2) --node[left]{$e_2$} (-.7, -1.3) node [below]{$v_1$};
        
        \draw (p4) to [bend left=10] node [above]{$e_6$}(p3);
        \draw (p3) to [bend left=50] node [below] {$e_6$}(p4);
        \node (p9) at (0, -4.1){$\bullet$};
        \node (p10) at (0, -3.5){$\bullet$};
        \draw (p4) to [bend right=12] node [above]{$e_5$} (0,-4.1) node [below]{$v_2$};
        \draw (0,-4.1) to [bend right=11] node[above]{$e_4$} (p3);
        \draw (0, -4.1) to [bend right = 16] node [below] {$e_3$} (p3);
        \draw (0,-4.1) -- node[right]{$e_2$} (0,-3.5) node [right]{$v_1$};
        
        \draw (p3) to [bend left=45] node [left]{$e_3$} (p2);
        \draw (p3) to [bend right = 11] node [left]{$e_3$}(p2);
        \draw (p3) to [bend right=17] node [right]{$e_4$} (p2);
        \node (p11) at (.8, -1.75){$\bullet$};
        \node (p12) at (.3, -1.5){$\bullet$};
        \draw (p2) to [bend right=3] node[right]{$e_5$} (.8, - 1.75) node[right]{$v_4$};
        \draw (.8, -1.75) to [bend right = 3] node [right]{$e_6$} (p3);
        \draw (p2) -- node [left]{$e_2$} (.3, -1.5) node [below]{$v_1$};
        
        \draw (p3) to [bend right=70] node[right]{$e_4$} (p2);
        \node(p13)at (1.1, -1){$\bullet$};
        \node(p14)at (1.9, -1.3) {$\bullet$};
        \draw (p2) -- node[right]{$e_2$} (1.1,-1) node [right]{$v_1$};
        \draw (p2) to [bend left=20] node [right]{$e_5$} (1.9, -1.3) node [right]{$v_4$};
        \draw (1.9, -1.3) to [bend left=20] node [left]{$e_6$} (p3);
        
        
        \draw[red, very thick] (b1) to [bend right = 50] node[right]{$\lambda_2$} (.3,.1);
        \draw[red,->, very thick] (.3,.1) arc (18.43:90:.3cm) -- (b2);

        \draw[red,->, very thick] (b1) to [out=330,in=300,looseness=8] node [right]{$\lambda_1$} (b1);
        \draw[red, very thick] (b1) to [bend left = 5] node [below] {$\lambda_3$} (2.3, -3.4);
        \draw[red,->, very thick] (2.3, -3.4) arc (18.43:143:.3cm) -- (b3);
        
        \draw[red, very thick] (b1) to [bend left = 1] node [above] {$\lambda_4$} (2.4, -3.3);
        \draw[red,->, very thick] (2.4, -3.3) arc (26.57: 90: .45) -- (b4);
        
        \draw[red, very thick] (b1) to [bend left = 60] node [below] {$\lambda_5$} (-2.05, -3.85);
        \draw[red,->, very thick] (-2.05, -3.85) arc (261.87: 420 : .354cm) -- (b5);
        
        \draw[red, very thick] (b1) to [bend left = 55] node [above] {$\lambda_6$} (-1.8, -3.95);
        \draw[red,->, very thick] (-1.8, -3.95) arc (296.57: 350 : .45cm) -- (b6);

        \end{tikzpicture}
        }

{\sf Figure 9.  Choosing connecting paths.}
        \end{center}

\newpage
Since the disks $D_j$ are disjoint from the postcritical set, the connecting paths $\lambda_j$ may take any route to $b_j$ once it has entered $D_j$.  So we may represent the connecting paths by drawing them on the same codomain plane as the graph $\GGG$ .  See Figure 10.  \\


\begin{center}
\resizebox{!}{300pt}{
        \begin{tikzpicture}  [scale=1.25]
        \node[draw, circle, fill=gray!8] (v1) at (0,4) {$v_1$};
        \node [draw, circle, fill=gray!8](v2) at (0,0){$v_2$};
        \node [draw, circle, fill=gray!8](v3) at (2,-3.5){$v_3$};
        \node [draw, circle, fill=gray!8](v4) at (-2,-3.5) {$v_4$};
        
        \node(b1) at (4,-3) {$\blacksquare$};
        \node(b2) at (0,2.5) {$\blacksquare$};
        \node(b5) at (-1.31,-2.3) {$\blacksquare$};
        \node(b6) at (-.5,-3.5) {$\blacksquare$};
        \node(b3) at (1.01, -2.2) {$\blacksquare$};
        \node(b4) at (1.4,-2) {$\blacksquare$};
        
        \node (bb2) at (.076, 2.5){};
        \node (bb3) at (1.09, -2.2){};
        \node(bb4) at (1.5, -2){};
        \node (bb5) at (-1.25, -2.3){};
        \node (bb6) at (-.6, -3.56){};
        
        \draw (4,-3) node [right]{$~b_1$};
        \draw (0,2.5) node [right]{$~b_2$};
        \draw (-1.31,-2.3) node [right]{~$b_5$};
        \draw (-.5,-3.5) node [below]{$b_6~$};
        \draw (1.01,-2.2) node [left]{$b_3~$};
        \draw (1.4,-2) node [right]{$~b_4$};

        \draw (v1) -- node[left]{$e_2$} (v2);

        \draw (v2) -- node[left]{$e_5$}(v4);
        
        \draw (v4) -- node[above]{$e_6$}(v3);
        
        \draw (v2) to [bend right = 10] node [left]{$e_3$}(v3);
        
        \draw (v2) to [bend left = 10] node [right]{$e_4$}(v3);
        
        \draw[red, very thick] (b1) to [bend right = 50] node[right]{$\lambda_2$} (.3795,.1265);
        \draw[->, red, very thick] (.3795,.1265) arc (18.43:78:.4cm) -- (bb2);
        
        \draw[red, very thick] (b1) to [bend left = 5] node [below] {$\lambda_3$} (2.3795, -3.4265);
        \draw[->, red, very thick] (2.3795, -3.4265) arc (18.43:118:.4cm) to [bend left = 3] (bb3);
        
        \draw[red, very thick] (b1) to [bend left = 1] node [above] {$\lambda_4$} (2.4, -3.3);
        \draw[->, red, very thick] (2.4, -3.3) arc (26.57: 100: .45) to [bend right=3] (bb4);
        
        \draw[red, very thick] (b1) to [bend left = 60] node [below] {$\lambda_5$} (-2.05, -3.85);
        \draw[->, red, very thick] (-2.05, -3.85) arc (261.87: 410 : .354cm) -- (bb5);
        
        \draw[red, very thick] (b1) to [bend left = 55] node [above] {$\lambda_6$} (-1.8, -3.95);
        \draw[->, red, very thick] (-1.8, -3.95) arc (296.57: 357 : .45cm) -- (bb6);

        \end{tikzpicture}}

{\sf Figure 10.  Representation of the connecting paths via paths in the codomain.}  
        \end{center}

We now compute the corresponding wreath recursion $\Phi$ for this example, choice of generators, labels of preimages, and choice of connecting paths.   For brevity, we  give only the computation of the wreath position $g_2|_5$ and informally describe several helpful observations. \\

Let us compute $g_2|_5$.  Note first that $5^{g_2}=3$.   We express the loop $[\lambda_5 * \tilde{g}_2[b_5] * \overline{\lambda}_3]$ as a word in the generators.   Here, $\tilde{g}_1[b_5]$ is the lift under $f$ of a loop at $b$ representing $g_1$ starting at $b_5 \in f^{-1}(b)$.

\begin{center}
\resizebox{!}{340pt}{
        \begin{tikzpicture}  [scale=1.5]
        \node[draw, circle, fill=gray!8] (v1) at (0,4) {$v_1$};
        \node [draw, circle, fill=gray!8](v2) at (0,0){$v_2$};
        \node [draw, circle, fill=gray!8](v3) at (2,-3.5){$v_3$};
        \node [draw, circle, fill=gray!8](v4) at (-2,-3.5) {$v_4$};
        
        \node(b1) at (4,-3) {$\blacksquare$};
        \node(b2) at (0,2.5) {$\blacksquare$};
        \node(b5) at (-1.31,-2.3) {$\blacksquare$};
        \node(b6) at (-.5,-3.5) {$\blacksquare$};
        \node(b3) at (1.01, -2.2) {$\blacksquare$};
        \node(b4) at (1.4,-2) {$\blacksquare$};
        
        \node (bb2) at (.076, 2.5){};
        \node (bb3) at (1.09, -2.2){};
        \node(bb4) at (1.5, -2){};
        \node (bb5) at (-1.25, -2.3){};
        \node (bb6) at (-.6, -3.56){};
        
        \draw (4,-3) node [right]{$~b_1$};
        \draw (0,2.5) node [right]{$~b_2$};
        \draw (-1.31,-2.3) node [right]{~$b_5$};
        \draw (-.5,-3.5) node [below]{$b_6~$};
        \draw (1.01,-2.2) node [left]{$b_3~$};
        \draw (1.4,-2) node [right]{$~b_4$};
        
        \draw (v1) -- node[left]{$e_2$} (v2);
        \draw (v2) -- node[left]{$e_5$}(v4);
        \draw (v4) -- node[above]{$e_6$}(v3);
        \draw (v2) to [bend right = 10] node [left]{$e_3$}(v3);
        \draw (v2) to [bend left = 10] node [right]{$e_4$}(v3);
        
        \draw [very thin] (-2, -3.73) arc (270:420:.23cm) -- (-.115, -0.19918584287042088875565632927318) arc (240: 290 : .23cm) to [bend right = 10] (1.857365, -3.317365) arc (130 : 0 : .23cm) to [bend right = 5] (4, -3) to [bend left = 60] (-2, -3.73) -- cycle [fill = darkgray, fill opacity=0.4];
        
        \draw [very thick, ->, red] (b1) to [bend left = 57] node [below] {$\lambda_5$}(-2, -3.73) arc (270:420:.23cm) -- (b5);
        \draw [very thick, ->, blue] (b5) -- node [left]{$\tilde{g}_{2_{b_5}}$}(-.115, -0.19918584287042088875565632927318) arc (240: 290 : .23cm) to [bend right = 5] (b3);
        \draw [very thick, ->, red] (b3) to [bend right = 4] (1.857365, -3.317365) arc (130 : 0 : .23cm) to [bend right = 5] node [above]{$\lambda_3$} (b1);
        \end{tikzpicture}}

{\sf Figure 11.  Computation of $g_2|_5$.}  
        \end{center}

By examining the enclosed region (shaded in Figure 11), we see that $\lambda_5 * \tilde{g}_{2_{b_5}} * \overline{\lambda}_{5^{\sigma(g_2)}}$ is homotopic to $g_3^{-1}$ and thus conclude $g_2|_5 = g_3^{-1}$.\\

Performing this computation for each wreath position $g_i |_j$, $1 \leq i \leq 4, 1 \leq j \leq 6$, we notice several striking patterns.  First, if $j$ is fixed by the monodromy $\sigma(g_i)$ then $g_i|_j$ is trivial.  The intuition is latent in the procedure we have described: if $v_k$ is the vertex of higher degree attached to the edge $e_j$, then $\lambda_j$ follows $g_k$ to $e_j$.  Since $j$ is fixed by the monodromy, the lift $\tilde{g}_i[b_j]$ stays within the disk obtained by blowing up $e_j$ and thus $\overline{\lambda}_{j^{g_i}} = \overline{\lambda}_j$ traverses $\lambda_j$ backwards, yielding a path homotopic to the constant map at $b$.  Together with our computation $g_2|_5 = g_3^{-1}$, this gives us 10 of the 24 wreath positions:
        \begin{align*}
        \Phi(g_1) &  = \langle ~~, ~~, ~1, ~1, ~1, ~1 \rangle (12)\\
        \Phi(g_2) &  = \langle ~~, ~~, ~~, ~~, ~g_3^{-1}, ~1\rangle (12534)\\
        \Phi(g_3) &  = \langle ~~, ~1, ~~, ~~, ~1, ~~  \rangle (1436)\\
        \Phi(g_4) &  = \langle ~~, ~1, ~1, ~1, ~~, ~~ \rangle (165)
        \end{align*}
The second observation we make is that if $j':=j^{g_i}\neq j$, if edges $e_j, e_{j'}$  are not parallel, and if the index of $v_i$ (at which both $e_j, e_{j'}$ are both incident) is largest among the three vertices comprised by the endpoints of $e_j$ and $e_j'$, then  the wreath position $g_i|_j$ is trivial.  In this case $\lambda_i$ follows $g_i$ to the edge $e_j$ and runs up $e_j$ to $b_j$.  Then the lift $\tilde{g}_i[b_j]$ travels back down $e_j$ to $g_i$, follows $g_i$ to $e_{j^{g_i}}$ and runs up $e_{j^{g_i}}$ to $b_{j^{g_i}}$.  Then since $v_i$ is the vertex of higher index attached to $e_{j^{g_i}}$, $\overline{\lambda}_{j^{g_i}}$ goes back down $e_{j^{g_i}}$ to $g_i$ and follows $g_i$ out backwards to $b$.  Thus $g_i|_j$ encircles no elements of $V$ and is therefore homotopic to the constant map.  This gives us five more wreath positions.
        \begin{align*}
        \Phi(g_1) &  = \langle ~~, ~~, ~1, ~1, ~1, ~1 \rangle (12)\\
        \Phi(g_2) &  = \langle ~1, ~~, ~~, ~~, ~g_3^{-1}, ~1\rangle (12534)\\
        \Phi(g_3) &  = \langle ~1, ~1, ~~, ~1, ~1, ~~  \rangle (1436)\\
        \Phi(g_4) &  = \langle ~1, ~1, ~1, ~1, ~~, ~1 \rangle (165)
        \end{align*}
Now we fill in the rest of the tabl1, encouraging the reader to check our work using the processes described above.
        \begin{align*}
        \Phi(g_1) &  = \langle ~g_1, ~1, ~1, ~1, ~1, ~1 \rangle (12)\\
        \Phi(g_2) &  = \langle ~1, ~g_2g_3, ~1, ~1, ~g_3^{-1}, ~1\rangle (12534)\\
        \Phi(g_3) &  = \langle ~1, ~1, ~g_3, ~1, ~1, ~1  \rangle (1436)\\
        \Phi(g_4) &  = \langle ~1, ~1, ~1, ~1, ~g_4, ~1 \rangle (165)\\\\
        \end{align*}

\newcommand{\etalchar}[1]{$^{#1}$}
\def\cprime{$'$}

%

\end{document}